\documentclass[12pt]{amsart}
\usepackage{graphicx}
\usepackage{amsmath}
\usepackage{amssymb}
\usepackage{setspace}
\usepackage{datetime}
\newtheorem{thm}{Theorem}[section]
\newtheorem{cor}[thm]{Corollary}

\newtheorem{lem}[thm]{Lemma}

\theoremstyle{definition}

\theoremstyle{remark}
\newtheorem{rem}[thm]{Remark}
\theoremstyle{conclusion}

\theoremstyle{question}
\numberwithin{equation}{section}

\begin{document}
\title[PDEs involving higher-order fractional Laplacians]{Super poly-harmonic properties, Liouville theorems and classification of nonnegative solutions to equations involving higher-order fractional Laplacians}

\author{Daomin Cao, Wei Dai, Guolin Qin}

\address{Institute of Applied Mathematics, Chinese Academy of Sciences, Beijing 100190, and University of Chinese Academy of Sciences, Beijing 100049, P. R. China}
\email{dmcao@amt.ac.cn}

\address{School of Mathematics and Systems Science, Beihang University (BUAA), Beijing 100083, P. R. China, and LAGA, Universit\'{e} Paris 13 (UMR 7539), Paris, France}
\email{weidai@buaa.edu.cn}

\address{Institute of Applied Mathematics, Chinese Academy of Sciences, Beijing 100190, and University of Chinese Academy of Sciences, Beijing 100049, P. R. China}
\email{qinguolin18@mails.ucas.ac.cn}

\thanks{D. Cao was supported by NNSF of China (No. 11771469) and Chinese Academy of Sciences (No. QYZDJ-SSW-SYS021). W. Dai was supported by the NNSF of China (No. 11971049), the Fundamental Research Funds for the Central Universities and the State Scholarship Fund of China (No. 201806025011).}

\begin{abstract}
In this paper, we are concerned with equations \eqref{PDE} involving higher-order fractional Laplacians. By introducing a new approach, we prove the super poly-harmonic properties for nonnegative solutions to \eqref{PDE} (Theorem \ref{Thm0}). Our theorem seems to be the first result on this problem. As a consequence, we derive many important applications of the super poly-harmonic properties. For instance, we establish Liouville theorems, integral representation formula and classification results for nonnegative solutions to fractional higher-order equations \eqref{PDE} with general nonlinearities $f(x,u,Du,\cdots)$ including conformally invariant and odd order cases. In particular, our results completely improve the classification results for third order equations in Dai and Qin \cite{DQ1} by removing the assumptions on integrability. We also derive a characterization for $\alpha$-harmonic functions via averages in the appendix.
\end{abstract}
\maketitle {\small {\bf Keywords:} Super poly-harmonic properties; Higher-order fractional Laplacians; Conformally invariant equations; Nonnegative classical solutions; Classification of solutions; Liouville theorems.\\

}

\section{Introduction}
\subsection{Background and setting of the problem}
In this paper, we mainly consider nonnegative classical solutions to the following equations involving higher-order fractional Laplacians
\begin{equation}\label{PDE}\\\begin{cases}
(-\Delta)^{m+\frac{\alpha}{2}}u(x)=f(x,u,Du,\cdots), \,\,\,\,\,\,\,\, x\in\mathbb{R}^{n}, \\
u\in C^{2m+[\alpha],\{\alpha\}+\epsilon}_{loc}\cap\mathcal{L}_{\alpha}(\mathbb{R}^{n}), \,\,\,\,\, u(x)\geq0, \,\,\,\,\, x\in\mathbb{R}^{n},
\end{cases}\end{equation}
where $n\geq2$, $1\leq m<+\infty$ is an integer, $0<\alpha<2$, $\epsilon>0$ is arbitrarily small, $[\alpha]$ denotes the integer part of $\alpha$, $\{\alpha\}:=\alpha-[\alpha]$, the higher-order fractional Laplacians $(-\Delta)^{m+\frac{\alpha}{2}}:=(-\Delta)^{m}(-\Delta)^{\frac{\alpha}{2}}$ and nonlinearity $f(x,u,Du,\cdots)\geq0$ is an arbitrary nonnegative function (may depend on $x$, $u$ and derivatives of $u$) which is continuous with respect to $x\in\mathbb{R}^{n}$.

For any $u\in C^{[\alpha],\{\alpha\}+\epsilon}_{loc}(\mathbb{R}^{n})\cap\mathcal{L}_{\alpha}(\mathbb{R}^{n})$, the nonlocal operator $(-\Delta)^{\frac{\alpha}{2}}$ ($0<\alpha<2$) is defined by (see \cite{CFY,CLL,DQ1,DQ2,S,ZCCY})
\begin{equation}\label{nonlocal defn}
  (-\Delta)^{\frac{\alpha}{2}}u(x)=C_{n,\alpha} \, P.V.\int_{\mathbb{R}^n}\frac{u(x)-u(y)}{|x-y|^{n+\alpha}}dy:=C_{n}\lim_{\varepsilon\rightarrow0}\int_{|y-x|\geq\varepsilon}\frac{u(x)-u(y)}{|x-y|^{n+\alpha}}dy,
\end{equation}
where the function space
\begin{equation}\label{0-1}
  \mathcal{L}_{\alpha}(\mathbb{R}^{n}):=\Big\{u: \mathbb{R}^{n}\rightarrow\mathbb{R}\,\big|\,\int_{\mathbb{R}^{n}}\frac{|u(x)|}{1+|x|^{n+\alpha}}dx<\infty\Big\}.
\end{equation}
The fractional Laplacians $(-\Delta)^{\frac{\alpha}{2}}$ can also be defined equivalently (see \cite{CLM}) by Caffarelli and Silvestre's extension method (see \cite{CS}) for $u\in C^{[\alpha],\{\alpha\}+\epsilon}_{loc}(\mathbb{R}^{n})\cap\mathcal{L}_{\alpha}(\mathbb{R}^{n})$. Throughout this paper, we define $(-\Delta)^{m+\frac{\alpha}{2}}u:=(-\Delta)^{m}(-\Delta)^{\frac{\alpha}{2}}u$ for $u\in C^{2m+[\alpha],\{\alpha\}+\epsilon}_{loc}(\mathbb{R}^{n})\cap\mathcal{L}_{\alpha}(\mathbb{R}^{n})$, where $(-\Delta)^{\frac{\alpha}{2}}u$ is defined by definition \eqref{nonlocal defn}. Due to the nonlocal feature of $(-\Delta)^{\frac{\alpha}{2}}$, we need to assume  $u\in C^{2m+[\alpha],\{\alpha\}+\epsilon}_{loc}(\mathbb{R}^{n})$ with arbitrarily small $\epsilon>0$ (merely $u\in C^{2m+[\alpha],\{\alpha\}}$ is not enough) to guarantee that $(-\Delta)^{\frac{\alpha}{2}}u\in C^{2m}(\mathbb{R}^{n})$ (see \cite{CLM,S}), and hence $u$ is a classical solution to equation \eqref{PDE} in the sense that $(-\Delta)^{m+\frac{\alpha}{2}}u$ is pointwise well-defined and continuous in the whole $\mathbb{R}^{n}$.

For $0<\gamma<+\infty$, PDEs of the form
\begin{equation}\label{GPDE}
  (-\Delta)^{\frac{\gamma}{2}}u(x)=f(x,u,\cdots)
\end{equation}
have numerous important applications in conformal geometry and Sobolev inequalities, which also model many phenomena in mathematical physics, astrophysics, probability and finance (see \cite{CDQ,CLL,CLM,CS,CY,DFHQW,Lin,WX,Xu} and the references therein). We say that equation \eqref{GPDE} is in critical order if $\gamma=n$, is in sub-critical order if $0<\gamma<n$ and is in super-critical order if $n<\gamma<+\infty$.

\subsection{Super poly-harmonic properties of nonnegative solutions}
First, we will investigate the super poly-harmonic properties of nonnegative solutions to \eqref{PDE}. It is well known that the super poly-harmonic properties of nonnegative solutions play a crucial role in establishing the integral representation formulae, Liouville type theorems and classification of solutions to higher order PDEs in $\mathbb{R}^{n}$ or $\mathbb{R}^{n}_{+}$ (see \cite{CD,CDQ,CF,CFL,CL2,DL,DPQ,DQ0,DQ1,DQ3,DQ5,DQZ,Lin,LWX,WX} and the references therein).

For integer higher-order equations (i.e., $\alpha=0$ in \eqref{PDE}), the super poly-harmonic properties of nonnegative solutions usually can be derived via the ``spherical average, re-centers and iteration" arguments in conjunction with careful ODE analysis (we refer to \cite{CFL,Lin,LWX,WX}, see also \cite{CDQ,CF,CL2,DQ0,DQ3,DQ5} and the references therein). However, \emph{for the fractional higher-order equation \eqref{PDE}, so far there is no result on the super poly-harmonic properties.} The reason for this is that $(-\Delta)^{\frac{\alpha}{2}}$ is nonlocal and $(-\Delta)^{\frac{\alpha}{2}}f(r)$ can not be calculated or expanded accurately ($0<\alpha<2$ and $f(r)$ is a radially symmetric function), thus the strategy for integer higher-order equations does not work any more for equation \eqref{PDE} involving higher-order fractional Laplacians. To overcome these difficulties we need to implement new ideas and arguments. In this paper, by taking full advantage of the Poisson representation formulae for $(-\Delta)^{\frac{\alpha}{2}}$ and developing some new integral estimates on the average $\int_{R}^{+\infty}\frac{R^{\alpha}}{r(r^{2}-R^{2})^{\frac{\alpha}{2}}}\bar{u}(r)dr$ and iteration techniques, we will introduce a new approach to overcome these difficulties and establish the super poly-harmonic properties of nonnegative classical solutions to \eqref{PDE} (see Section 2). Our theorem seems to be the first result on this problem.
\begin{thm}\label{Thm0}
Assume $n\geq2$, $m\geq 1$, $0<\alpha<2$ and $f\geq0$ is continuous w.r.t. $x\in\mathbb{R}^{n}$. Suppose that $u$ is a nonnegative classical solution to \eqref{PDE}. Then, we have, for every $i=0,1,\cdots,m-1$,
\begin{equation}\label{0-2}
  (-\Delta)^{i+\frac{\alpha}{2}}u(x)\geq0, \qquad \forall \,\, x\in\mathbb{R}^{n}.
\end{equation}
\end{thm}

Now suppose the nonlinearity $f(x,u,Du,\cdots)\leq0$ in \eqref{PDE} is an arbitrary function (may depend on $x$, $u$ and derivatives of $u$) which is continuous with respect to $x\in\mathbb{R}^{n}$, by using the ideas in proving the super poly-harmonic properties in Theorem \ref{Thm0}, we can derive the following sub poly-harmonic properties of nonnegative classical solutions to \eqref{PDE}.
\begin{thm}\label{Thm1}
Assume $n\geq2$, $m\geq 1$, $0<\alpha<2$ and $f\leq0$ is continuous w.r.t. $x\in\mathbb{R}^{n}$. Suppose that $u$ is a nonnegative classical solution to \eqref{PDE}. Then, we have, for every $i=0,1,\cdots,m-1$,
\begin{equation}\label{0-3}
  (-\Delta)^{i+\frac{\alpha}{2}}u(x)\leq0, \qquad \forall \,\, x\in\mathbb{R}^{n}.
\end{equation}
\end{thm}

\subsection{Liouville theorems, integral representation formula and classification of nonnegative solutions}
In this subsection, by applying the super poly-harmonic properties in Theorem \ref{Thm0}, we will derive some important results on equations involving higher-order fractional Laplacians.

{\em (i)  Liouville theorem for fractional poly-harmonic functions in $\mathbb{R}^{n}$.}
\smallskip

Assume $u\geq0$ is a nonnegative fractional poly-harmonic functions in $\mathbb{R}^{n}$, that is,
\begin{equation}\label{0-4}
  (-\Delta)^{m+\frac{\alpha}{2}}u(x)=0, \qquad \forall \,\, x\in\mathbb{R}^{n},
\end{equation}
where $n\geq2$, $1\leq m<+\infty$ is an integer and $0<\alpha<2$.

As a consequence of the super poly-harmonic properties in Theorem \ref{Thm0} and the sub poly-harmonic properties in Theorem \ref{Thm1}, we deduce that $(-\Delta)^{i+\frac{\alpha}{2}}u\equiv0$ in $\mathbb{R}^{n}$ for every $i=0,\cdots,m-1$. In particular, one has $(-\Delta)^{\frac{\alpha}{2}}u\equiv0$ in $\mathbb{R}^{n}$, and hence from the Liouville theorem for fractional Laplacians $(-\Delta)^{\frac{\alpha}{2}}$ with $0<\alpha<2$ in \cite{BKN,ZCCY}, it follows that $u\equiv C$ in $\mathbb{R}^{n}$ for some nonnegative constant $C\geq0$. Therefore, we have the following Liouville theorem for fractional poly-harmonic functions in $\mathbb{R}^{n}$.
\begin{thm}\label{Thm4}
Assume $n\geq2$, $m\geq1$ and $0<\alpha<2$. Suppose $u$ is a nonnegative fractional poly-harmonic functions in $\mathbb{R}^{n}$ satisfying \eqref{0-4}, then $u\equiv C\geq0$ in $\mathbb{R}^{n}$.
\end{thm}

{\em (ii)  Subcritical order cases $2m+\alpha<n$.}
\smallskip

Equation \eqref{PDE} is closely related to the following integral equation
\begin{equation}\label{IE}
  u(x)=\int_{\mathbb{R}^{n}}\frac{R_{2m+\alpha,n}}{|x-y|^{n-2m-\alpha}}f(y,u(y),\cdots)dy,
\end{equation}
where the Riesz potential's constants $R_{\gamma,n}:=\frac{\Gamma(\frac{n-\gamma}{2})}{\pi^{\frac{n}{2}}2^{\gamma}\Gamma(\frac{\gamma}{2})}$ for $0<\gamma<n$ (see \cite{Stein}).

From the super poly-harmonic properties of nonnegative solutions in Theorem \ref{Thm0}, by using the methods in \cite{CFY,ZCCY}, we can deduce the following equivalence between PDEs \eqref{PDE} and IEs \eqref{IE}.
\begin{thm}\label{equivalence}
Assume $2m+\alpha<n$, $m\geq1$, $0<\alpha<2$ and $f\geq0$ is continuous w.r.t. $x\in\mathbb{R}^{n}$. Suppose that $u$ is a nonnegative classical solution to \eqref{PDE}, then $u$ is also a nonnegative solution to integral equation \eqref{IE}, and vice versa.
\end{thm}
\begin{rem}\label{remark1}
Based on Theorem \ref{Thm0}, the proof of Theorem \ref{equivalence} is entirely similar to \cite{CFY,ZCCY} (see also \cite{DQ0,DQ2}), so we omit the details here.
\end{rem}
\begin{rem}\label{remark0}
One can observe that, Theorem \ref{Thm0} and Theorem \ref{equivalence} hold for PDEs \eqref{PDE} and IEs \eqref{IE} if we take the nonlinearities $f=|x|^{a}u^{p}$ ($a\geq0$, $p>0$), $f=|x|^{a}e^{nu}$ ($a\geq0$) or $f=|x|^{a}u^{p}(1+|\nabla u|^{2})^{\frac{\kappa}{2}}$ ($a\geq0$, $p>0$, $\kappa>0$) and so on $\cdots$. If we consider positive solution $u>0$, then Theorem \ref{Thm0} and Theorem \ref{equivalence} are also valid for PDEs \eqref{PDE} and IEs \eqref{IE} with $f=|x|^{a}u^{-q}$ ($a\geq0$, $q>0$).
\end{rem}

Based on the equivalence between PDEs \eqref{PDE} and IEs \eqref{IE}, we will first consider the conformally invariant case $f=u^{\frac{n+2m+\alpha}{n-2m-\alpha}}$, which is geometrically interesting.

The quantitative and qualitative properties of solutions to fractional order or higher order conformally invariant equations of the form
\begin{equation}\label{PDE-CI}
  (-\Delta)^{\frac{\gamma}{2}}u=u^{\frac{n+\gamma}{n-\gamma}} \qquad \text{with} \,\,\, 0<\gamma<n
\end{equation}
have been extensively studied (see \cite{CGS,CLL,CLO,CY,DQ0,DQ1,GNN,Lin,Li,WX,Xu} and the references therein). The classification results for conformally invariant equations \eqref{PDE-CI} have important applications in many problems from conformal geometry (i.e., prescribing scalar curvature problems, variational problems involving Paneitz operators on compact Riemannian manifolds, see \cite{CGS,CL0,CL3,CL1,CLO,CY,DQ1,Lin,Li,LWX,LZ,WX,Xu}). In \cite{CLO}, by developing the method of moving planes in integral forms, Chen, Li and Ou classified all positive $L^{\frac{2n}{n-\gamma}}_{loc}$ solutions to the equivalent integral equation of  PDE \eqref{PDE-CI} for general $\gamma\in(0,n)$. As a consequence, they obtained the classification for positive weak solutions to PDE \eqref{PDE-CI}. As to the classification theorems for positive classical solutions to PDE \eqref{PDE-CI}, all known results are focused on the cases that $0<\gamma<2$, or $2\leq\gamma<n$ is an even integer (see Caffarelli, Gidas and Spruck \cite{CGS}, Chen, Li and Li \cite{CLL}, Chen, Li and Ou \cite{CLO}, Gidas, Ni and Nirenberg \cite{GNN}, Lin \cite{Lin}, Wei and Xu \cite{WX}).

One should observe that, when $\gamma\in(2,n)$ is an odd integer, or more general, when $\gamma=2m+\alpha<n$ with $m\geq1$ and $0<\alpha<2$, classification for positive classical solutions to \eqref{PDE-CI} is still open. In the particular case $\gamma=3$, by applying the harmonic asymptotic expansions for $(-\Delta)^{\frac{1}{2}}\bar{u}$ ($\bar{u}$ is the Kelvin transform of $u$) and the method of moving planes to the third-order equation \eqref{PDE-CI} directly, Dai and Qin \cite{DQ1} derived the classification of nonnegative classical solutions to \eqref{PDE-CI} under additional weak integrability assumption $\int_{\mathbb{R}^{n}}\frac{u^{\frac{n+3}{n-3}}}{|x|^{n-3}}dx<\infty$.

In this paper, by the classification of positive $L^{\frac{2n}{n-2m-\alpha}}_{loc}$ solutions to integral equation \eqref{IE} in \cite{CLO} (Theorem 1 in \cite{CLO}) and the equivalence between PDE \eqref{PDE} and integral equation \eqref{IE} in Theorem \ref{equivalence}, we can classify all positive classical solutions to \eqref{PDE} in the conformally invariant cases $f=u^{\frac{n+2m+\alpha}{n-2m-\alpha}}$ without any assumptions on integrability or decay of $u$.

Our classification result for \eqref{PDE} in the conformally invariant cases is as follows.
\begin{thm}\label{classification}
Assume $2m+\alpha<n$, $m\geq1$, $0<\alpha<2$ and $f=u^{\frac{n+2m+\alpha}{n-2m-\alpha}}$. Suppose that $u$ is a nonnegative classical solution of \eqref{PDE}, then either $u\equiv0$ or $u$ is of the following form
\begin{equation*}
  u(x)=\mu^{\frac{n-2m-\alpha}{2}}Q\big(\mu(x-x_{0})\big) \,\,\,\,\,\,\,\,\,\,\,\, \text{for some} \,\, \mu>0 \,\,\, \text{and} \,\,\, x_{0}\in\mathbb{R}^{n},
\end{equation*}
where
\begin{equation*}
  Q(x):=\left(\frac{1}{R_{2m+\alpha,n}I\big(\frac{n-2m-\alpha}{2}\big)}\right)^{\frac{n-2m-\alpha}{2(2m+\alpha)}}\Big(\frac{1}{1+|x|^{2}}\Big)^{\frac{n-2m-\alpha}{2}}
\end{equation*}
with $I(s):=\frac{\pi^{\frac{n}{2}}\Gamma\big(\frac{n-2s}{2}\big)}{\Gamma(n-s)}$ for $0<s<\frac{n}{2}$.
\end{thm}
\begin{rem}\label{remark2}
Theorem \ref{classification} follows directly from Theorem 1 in \cite{CLO} and Theorem \ref{equivalence}, so we omit the details here. The exact constants in the expression of $Q(x)$ are given by formula (37) in Lemma 4.1 in \cite{DFHQW}.
\end{rem}

\begin{rem}\label{remark3}
Combining Theorem \ref{classification} with the classification theorems in \cite{CGS,CLL,CLO,GNN,Lin,WX} gives us the complete classification results for conformally invariant equations \eqref{PDE-CI} in all the cases $0<\gamma<n$. If we take $\alpha=1$, then Theorem \ref{classification} gives the classification results for all the odd order conformally invariant equations \eqref{PDE}. In particular, Theorem \ref{classification} completely improves the classification results for third order conformally invariant equations in \cite{DQ1} by removing the integrability assumption $\int_{\mathbb{R}^{n}}\frac{u^{\frac{n+3}{n-3}}}{|x|^{n-3}}dx<\infty$.
\end{rem}

Next, we take $f=|x|^{a}u^{p}$ ($a\geq0$, $p>0$) and study the Liouville property of nonnegative solutions in the subcritical cases.

For PDEs \eqref{PDE} and IEs \eqref{IE}, we say the Hardy-H\'{e}non type nonlinearities $f=|x|^{a}u^{p}$ is subcritical if $0<p<p_{c}(a):=\frac{n+2m+\alpha+2a}{n-2m-\alpha}$, critical if $p=p_{c}(a)$ and super-critical if $p>p_{c}(a)$. There are also lots of literature on Liouville type theorems for fractional order or higher order Hardy-H\'{e}non type equations in the subcritical cases, and we refer to \cite{CF,CFL,CFY,CGS,CLL,DL,DPQ,DQ1,DQ2,DQ4,DQ5,DQZ,Lin,LZ1,WX,ZCCY} and the references therein. It should be noted that, all the known results focused on the cases $m=0$ or $\alpha=0$, hence Liouville type theorems for general fractional higher-order cases $m\geq1$ and $0<\alpha<2$ are still open. In the particular case $m=\alpha=1$ and $f=u^{p}$ with $1\leq p<\frac{n+3}{n-3}$, Dai and Qin \cite{DQ1} derived Liouville type theorem for nonnegative classical solutions to \eqref{PDE} under additional weak integrability assumption $\int_{\mathbb{R}^{n}}\frac{u^{p}}{|x|^{n-3}}dx<\infty$.

In this paper, by applying the method of scaling spheres developed recently by Dai and Qin \cite{DQ2} (see also \cite{DQ3,DQ4,DQZ}), we will establish Liouville type theorem for nonnegative solutions to IEs \eqref{IE}. Our Liouville type result for IEs \eqref{IE} is as follows.
\begin{thm}\label{Thm2}
Assume $2m+\alpha<n$, $m\geq1$, $0<\alpha<2$ and $f=|x|^{a}u^{p}$ with $a\geq0$ and $0<p<p_{c}(a)$. Suppose $u\in C(\mathbb{R}^{n})$ is a nonnegative solution to IEs \eqref{IE}, then $u\equiv0$ in $\mathbb{R}^{n}$.
\end{thm}

\begin{rem}\label{remark4}
It is clear from the proof of Theorem \ref{Thm2} that (see \eqref{3-38} in Section 3), the Liouville type results in Theorem \ref{Thm2} are also valid for $f=|x_{i}|^{a}u^{p}$ ($i=1,2,\cdots,n$) with $a\geq0$ and $0<p<p_{c}(a)$. Theorem \ref{Thm2} can also be available for more general nonlinearities $f(x,u)$ satisfying appropriate assumptions, we leave the details to readers (we refer to \cite{DQ2,DQ3,DQ4,DQZ}).
\end{rem}

From the equivalence between PDEs \eqref{PDE} and IEs \eqref{IE} in Theorem \ref{equivalence} and Theorem \ref{Thm2}, we derive the following Liouville type result for nonnegative classical solutions to PDEs \eqref{PDE} immediately.
\begin{cor}\label{Cor0}
Assume $2m+\alpha<n$, $m\geq1$, $0<\alpha<2$ and $f(x,u)=|x|^{a}u^{p}$ with $a\geq0$ and $0<p<p_{c}(a)$. Suppose $u$ is a nonnegative classical solution to PDEs \eqref{PDE}, then $u\equiv0$ in $\mathbb{R}^{n}$.
\end{cor}

\begin{rem}\label{remark5}
If we take $\alpha=1$, then Corollary \ref{Cor0} gives Liouville type results for all the odd order equations \eqref{PDE} with $f=|x|^{a}u^{p}$ in subcritical cases $0<p<p_{c}(a)$. In particular, Corollary \ref{Cor0} completely improves the Liouville theorem for third order equations \eqref{PDE} with $f=u^{p}$ ($1\leq p<\frac{n+3}{n-3}$) in \cite{DQ1} by removing the integrability assumption $\int_{\mathbb{R}^{n}}\frac{u^{p}}{|x|^{n-3}}dx<\infty$ and extending $1\leq p<\frac{n+3}{n-3}$ to the full subcritical range $0<p<\frac{n+3}{n-3}$.
\end{rem}

{\em (iii)  Critical and super-critical order cases: $n\leq 2m+\alpha<+\infty$.}
\smallskip

As an immediate consequence of the super poly-harmonic properties in Theorem \ref{Thm0}, by arguments developed by Chen, Dai and Qin \cite{CDQ}, we can establish Liouville type theorem for nonnegative solutions to \eqref{PDE} with general nonlinearities $f$ in both critical and super-critical order cases. For the particular case $\alpha=0$, Liouville type theorems for integer higher-order H\'{e}non-Hardy type equations in $\mathbb{R}^{n}$ or $\mathbb{R}^{n}_{+}$ have been derived by Chen, Dai and Qin \cite{CDQ} and Dai and Qin \cite{DQ3} in both critical and super-critical order cases. Our result will extend the results in \cite{CDQ} to general fractional higher-order cases $0<\alpha<2$ and general nonlinearities $f(x,u,\cdots)$.

For the critical and super-critical order cases we have the following result.
\begin{thm}\label{Thm3}
Assume $n\geq3$, $m\geq1$, $0<\alpha<2$, $\frac{n}{2}\leq m+\frac{\alpha}{2}<+\infty$, $f\geq0$ is continuous w.r.t. $x\in\mathbb{R}^{n}$ and $f>0$ at some point in $\mathbb{R}^{n}$ if $u>0$ in the whole $\mathbb{R}^{n}$. Suppose that $u$ is a nonnegative classical solution to \eqref{PDE}, then $u\equiv0$ in $\mathbb{R}^{n}$.
\end{thm}

\begin{rem}\label{remark6}
If we take $\alpha=1$, then Theorem \ref{Thm3} gives Liouville type results for all the critical and super-critical order equations \eqref{PDE} involving odd order Laplacians. One should observe that, if $f\geq0$ is continuous w.r.t. $x\in\mathbb{R}^{n}$ and $f\geq C|x|^{a}u^{p}$ for some $a\in\mathbb{R}$, $p>0$, $C>0$ and some point $x\neq0$ in $\mathbb{R}^{n}$, then $f$ satisfies the assumptions in Theorem \ref{Thm3} and hence Theorem \ref{Thm3} is valid for equations \eqref{PDE} with such kind of nonlinearities $f(x,u,\cdots)$.
\end{rem}

\begin{rem}\label{remark7}
If we consider positive solution $u>0$, suppose $f\geq0$ is continuous w.r.t. $x\in\mathbb{R}^{n}$ and $f\geq C|x|^{a}u^{-q}$ for some $a\in\mathbb{R}$, $q>0$, $C>0$ and some point $x\neq0$ in $\mathbb{R}^{n}$, then Theorem \ref{Thm3} implies nonexistence of positive solutions and thus extend Theorems 1.2 and 1.3 in \cite{N} to general fractional higher-order cases $0<\alpha<2$ and general nonlinearities $f(x,u,\cdots)$.
\end{rem}

This paper is organized as follows. In Section 2, we will carry out our proof of Theorem \ref{Thm0}. In Section 3, we will prove Theorem \ref{Thm1}. Section 4 and 5 are devoted to proving Theorem \ref{Thm2} and \ref{Thm3} respectively. In the Appendix, we establish an important characterization for $\alpha$-harmonic functions via the averages $\int_{R}^{+\infty}\frac{R^{\alpha}}{r(r^{2}-R^{2})^{\frac{\alpha}{2}}}\overline{u}(r)dr$ and deduce some important properties for $\alpha$-harmonic functions.

Throughout this paper, we will use $C$ to denote a general positive constant that may depend on $u$ and the quantities appearing in the subscript, and whose value may differ from line to line.

\section{Proof of Theorem \ref{Thm0}}
In this section, we will carry out our proof of the super poly-harmonic properties for nonnegative solutions to \eqref{PDE} (i.e., Theorem \ref{Thm0}) via contradiction arguments.

Let $v_{i}:=(-\Delta)^{i+\frac{\alpha}{2}}u$ for $i=0,1,\cdots,m-1$, then it follows from equation \eqref{PDE} that
\begin{equation}\label{2-0}
\left\{\begin{array}{l}{(-\Delta)^{\frac{\alpha}{2}}u=v_{0} \quad \text{in} \,\, \mathbb{R}^{n},} \\ {-\Delta v_{0}=v_{1} \quad \text{in} \,\, \mathbb{R}^{n},} \\ {\cdots\cdots} \\{-\Delta v_{m-1}=f\geq0 \quad \text{in} \,\, \mathbb{R}^{n}.}\end{array}\right.
\end{equation}

Suppose that Theorem \ref{Thm0} does not hold, then there must exist a largest integer $0\leq k\leq m-1$ and a point $x_{0}\in\mathbb{R}^{n}$ such that
\begin{equation}\label{2-1}
v_{k}\left(x_{0}\right)=(-\Delta)^{k+\frac{\alpha}{2}} u\left(x_{0}\right)<0.
\end{equation}
Let
\begin{equation}\label{2-3'}
  \bar{g}(r)=\bar{g}\big(|x-x_{0}|\big):=\frac{1}{|\partial B_{r}(x_{0})|}\int_{\partial B_{r}(x_{0})}g(x)d\sigma
\end{equation}
be the spherical average of a function $g$ with respect to the center $x_{0}$.

First, we will show that $0\leq k\leq m-1$ is even. Suppose not, assume $k$ is an odd integer. From \eqref{2-0} and the well-known property $\overline{\Delta u}=\Delta\bar{u}$, we get
\begin{equation}\label{2-2}
\overline{v_{k}}(r)\leq\overline{v_{k}}(0):=-c_{0}<0, \qquad \forall \,\, r>0.
\end{equation}
It follows immediately that
\begin{equation}\label{2-3}
\overline{v_{k-1}}(r)\geq\overline{v_{k-1}}(0)+\frac{c_{0}}{2n}r^{2}, \qquad \forall \,\, r>0,
\end{equation}
and
\begin{equation}\label{2-4}
\overline{v_{k-2}}(r)\leq\overline{v_{k-2}}(0)-\frac{r^{2}}{2n}\overline{v_{k-1}}(0)-\frac{c_{0}}{8n(n+2)}r^{4}, \qquad \forall \,\, r>0.
\end{equation}
Repeating the above argument, we get
\begin{equation}\label{2-5}
  \overline{v_{0}}(r)\geq \overline{v_{0}}(0)+c_{1}r^{2}+c_{2}r^{4}+\cdots+c_{k}r^{2k}, \qquad \forall \,\, r>0,
\end{equation}
where $c_{k}>0$. From \eqref{2-5}, we infer that there exists a $r_{0}$ large enough, such that
\begin{equation}\label{2-6}
\overline{v_{0}}(r)\geq\frac{1}{2}c_{k}r^{2k}, \qquad \forall \,\, r>r_{0}.
\end{equation}
From the first equation in \eqref{2-0}, we conclude that, for arbitrary $R>0$,
\begin{equation}\label{2-7}
  u(x)=\int_{B_{R}(x_{0})}G^\alpha_R(x,y)v_{0}(y)dy+\int_{|y-x_{0}|>R}P^{\alpha}_{R}(x,y)u(y)dy, \qquad \forall \,\, x\in B_{R}(x_{0}),
\end{equation}
where the Green's function for $(-\Delta)^{\frac{\alpha}{2}}$ with $0<\alpha<2$ on $B_R(x_{0})$ is given by
\begin{equation}\label{2-8}
G^\alpha_R(x,y):=\frac{C_{n,\alpha}}{|x-y|^{n-\alpha}}\int_{0}^{\frac{t_{R}}{s_{R}}}\frac{b^{\frac{\alpha}{2}-1}}{(1+b)^{\frac{n}{2}}}db
\,\,\,\,\,\,\,\,\, \text{if} \,\, x,y\in B_{R}(x_{0})
\end{equation}
with $s_{R}=\frac{|x-y|^{2}}{R^{2}}$, $t_{R}=\left(1-\frac{|x-x_{0}|^{2}}{R^{2}}\right)\left(1-\frac{|y-x_{0}|^{2}}{R^{2}}\right)$, and $G^{\alpha}_{R}(x,y)=0$ if $x$ or $y\in\mathbb{R}^{n}\setminus B_{R}(x_{0})$ (see \cite{K}), and the Poisson kernel $P^{\alpha}_{R}(x,y)$ for $(-\Delta)^{\frac{\alpha}{2}}$ in $B_{R}(x_{0})$ is defined by $P^{\alpha}_{R}(x,y):=0$ for $|y-x_{0}|<R$ and
\begin{equation}\label{2-9}
  P^{\alpha}_{R}(x,y):=\frac{\Gamma(\frac{n}{2})}{\pi^{\frac{n}{2}+1}}\sin\frac{\pi\alpha}{2}\left(\frac{R^{2}-|x-x_{0}|^{2}}{|y-x_{0}|^{2}-R^{2}}\right)^{\frac{\alpha}{2}}\frac{1}{|x-y|^{n}}
\end{equation}
for $|y-x_{0}|>R$ (see \cite{CLM}). Therefore, we have
\begin{eqnarray}\label{2-10}
  &&+\infty>u(x_{0})=\int_{B_{R}(x_{0})}\frac{C_{n,\alpha}}{|y-x_{0}|^{n-\alpha}}\left(\int_{0}^{\frac{R^{2}}{|y-x_{0}|^{2}}-1}\frac{b^{\frac{\alpha}{2}-1}}{(1+b)^{\frac{n}{2}}}db\right)v_{0}(y)dy \\
  \nonumber &&\qquad\qquad\qquad\,\,\,+C'_{n,\alpha}\int_{|y-x_{0}|>R}\frac{R^{\alpha}}{(|y-x_{0}|^{2}-R^{2})^{\frac{\alpha}{2}}}\cdot\frac{u(y)}{|y-x_{0}|^{n}}dy \\
 \nonumber &&\qquad=C_{n,\alpha}\int_{0}^{R}r^{\alpha-1}\left(\int_{0}^{\frac{R^{2}}{r^{2}}-1}\frac{b^{\frac{\alpha}{2}-1}}{(1+b)^{\frac{n}{2}}}db\right)\overline{v_{0}}(r)dr
  +C'_{n,\alpha}\int_{R}^{+\infty}\frac{R^{\alpha}}{r(r^{2}-R^{2})^{\frac{\alpha}{2}}}\overline{u}(r)dr.
\end{eqnarray}
Observe that, if $0<r\leq\frac{R}{2}$, then $3\leq\frac{R^{2}}{r^{2}}-1<+\infty$, and hence
\begin{equation}\label{2-11}
  \int_{0}^{3}\frac{b^{\frac{\alpha}{2}-1}}{(1+b)^{\frac{n}{2}}}db
  \leq\int_{0}^{\frac{R^{2}}{r^{2}}-1}\frac{b^{\frac{\alpha}{2}-1}}{(1+b)^{\frac{n}{2}}}db
  \leq\int_{0}^{+\infty}\frac{b^{\frac{\alpha}{2}-1}}{(1+b)^{\frac{n}{2}}}db.
\end{equation}
As a consequence of \eqref{2-5}, \eqref{2-6}, \eqref{2-10} and \eqref{2-11}, we deduce that
\begin{eqnarray}\label{2-12}
  u(x_{0})&\geq&C_{n,\alpha}\int_{r_{0}}^{\frac{R}{2}}r^{\alpha-1}\overline{v_{0}}(r)dr-\widetilde{C}_{n,\alpha}\int_{0}^{r_{0}}r^{\alpha-1}|\overline{v_{0}}(r)|dr \\
 \nonumber &\geq&C\int_{r_{0}}^{\frac{R}{2}}r^{2k+\alpha-1}dr-\widetilde{C}\geq CR^{2k+\alpha}-\widetilde{C}
\end{eqnarray}
for any $R>2r_{0}$. By letting $R\rightarrow+\infty$ in \eqref{2-12}, we get immediately a contradiction. Therefore, $k$ must be even.

Next, we will show that $k=0$. Suppose on contrary that $2\leq k\leq m-1$ is even, through similar procedure as in deriving \eqref{2-5}, we obtain
\begin{equation}\label{2-13}
  \overline{v_{0}}(r)\leq\overline{v_{0}}(0)-c_{1}r^{2}-c_{2}r^{4}-\cdots-c_{k}r^{2k}, \qquad \forall \,\, r>0,
\end{equation}
where $c_{k}>0$. Thus there exists a $r_{1}>0$ large enough such that
\begin{equation}\label{2-14}
\overline{v_{0}}(r)\leq-\frac{1}{2}c_{k}r^{2 k}, \qquad \forall \,\, r>r_{1}.
\end{equation}
Observe that, if $\frac{R}{2}<r<R$, then $0<\frac{R^{2}}{r^{2}}-1<3$, and hence
\begin{equation}\label{2-15}
  \int_{0}^{\frac{R^{2}}{r^{2}}-1}\frac{b^{\frac{\alpha}{2}-1}}{(1+b)^{\frac{n}{2}}}db
  \geq\int_{0}^{\frac{R^{2}}{r^{2}}-1}\frac{b^{\frac{\alpha}{2}-1}}{2^{n}}db\geq C_{n,\alpha}\left(\frac{R^{2}}{r^{2}}-1\right)^{\frac{\alpha}{2}}.
\end{equation}
It follows from \eqref{2-10}, \eqref{2-11}, \eqref{2-13}, \eqref{2-14} and \eqref{2-15} that, for any $R>2r_{1}$,
\begin{eqnarray}\label{2-16}
 && \int_{R}^{+\infty}\frac{R^{\alpha}\overline{u}(r)}{r(r^{2}-R^{2})^{\frac{\alpha}{2}}}dr\geq-C_{n,\alpha}
  \int_{0}^{R}r^{\alpha-1}\left(\int_{0}^{\frac{R^{2}}{r^{2}}-1}\frac{b^{\frac{\alpha}{2}-1}}{(1+b)^{\frac{n}{2}}}db\right)\overline{v}_{0}(r)dr \\
 \nonumber &\geq&C\int_{r_{1}}^{\frac{R}{2}}r^{2k+\alpha-1}dr-\widetilde{C}\int_{0}^{r_{1}}r^{\alpha-1}|\overline{v}_{0}(r)|dr
  +C\int_{\frac{R}{2}}^{R}r^{2k+\alpha-1}\left(\frac{R^{2}}{r^{2}}-1\right)^{\frac{\alpha}{2}}dr \\
  \nonumber &\geq& CR^{2k+\alpha}-\widetilde{C}.
\end{eqnarray}
Thus there exists a $r_{2}>2r_{1}$ large enough such that
\begin{equation}\label{2-17}
\int_{R}^{+\infty}\frac{R^{\alpha}\overline{u}(r)}{r\left(r^{2}-R^{2}\right)^{\frac{\alpha}{2}}}dr\geq CR^{2k+\alpha}, \qquad \forall \,\, R>r_{2}.
\end{equation}
Since $u\in\mathcal{L}_{\alpha}(\mathbb{R}^{n})$, we have
\begin{equation}\label{2-18}
  \int_{|x-x_{0}|>1}\frac{u(x)}{|x-x_{0}|^{n+\alpha}}dx=C\int^{+\infty}_{1}\frac{\overline{u}(r)}{r^{1+\alpha}}dr<+\infty,
\end{equation}
and hence, for any $\delta>0$,
\begin{eqnarray}\label{2-19}
  &&\int_{1}^{+\infty}\frac{1}{R^{1+\alpha+\delta}}\int_{R}^{+\infty}\frac{R^{\alpha}\overline{u}(r)}{r(r^{2}-R^{2})^{\frac{\alpha}{2}}}drdR
  =\int_{1}^{+\infty}\frac{\overline{u}(r)}{r}\int_{1}^{r}\frac{1}{R^{1+\delta}(r^{2}-R^{2})^{\frac{\alpha}{2}}}dRdr \\
 \nonumber &\leq&C\int_{1}^{+\infty}\frac{\overline{u}(r)}{r^{1+\alpha}}\int_{1}^{\frac{r}{2}}\frac{1}{R^{1+\delta}}dRdr
  +C\int_{1}^{+\infty}\frac{\overline{u}(r)}{r^{2+\delta}}\int^{r}_{\frac{r}{2}}\frac{1}{r^{\frac{\alpha}{2}}(r-R)^{\frac{\alpha}{2}}}dRdr \\
 \nonumber &\leq& C_{\delta}\int_{1}^{+\infty}\frac{\overline{u}(r)}{r^{1+\alpha}}dr+C\int_{1}^{+\infty}\frac{\overline{u}(r)}{r^{1+\alpha}}dr<+\infty,
\end{eqnarray}
which is a contradiction with \eqref{2-17} and thus $k=0$.

Since $k=0$, we deduce that
\begin{equation}\label{2-20}
  \overline{v_{0}}(r)\leq\overline{v_{0}}(0):=-c_{0}<0, \qquad \forall \,\, r>0.
\end{equation}
Thus \eqref{2-10}, \eqref{2-11}, \eqref{2-15} and \eqref{2-20} yield that, for any $R>0$,
\begin{eqnarray}\label{2-21}
 && \int_{R}^{+\infty}\frac{R^{\alpha}\overline{u}(r)}{r(r^{2}-R^{2})^{\frac{\alpha}{2}}}dr\geq C\int^{\frac{R}{2}}_{0}r^{\alpha-1}dr
 +C\int^{R}_{\frac{R}{2}}r^{\alpha-1}\left(\frac{R^{2}}{r^{2}}-1\right)^{\frac{\alpha}{2}}dr \\
 \nonumber &\geq&CR^{\alpha}+C\int_{\frac{R}{2}}^{R}R^{\frac{\alpha}{2}-1}\left(R-r\right)^{\frac{\alpha}{2}}dr\geq CR^{\alpha}.
\end{eqnarray}
Since $u\in\mathcal{L}_{\alpha}(\mathbb{R}^{n})$, we have
\begin{equation}\label{2-22}
  \int_{|x-x_{0}|>N}\frac{u(x)}{|x-x_{0}|^{n+\alpha}}dx=C\int^{+\infty}_{N}\frac{\overline{u}(r)}{r^{1+\alpha}}dr=o_{N}(1)
\end{equation}
as $N\rightarrow+\infty$, and hence
\begin{equation}\label{2-23}
  \int_{2R}^{+\infty}\frac{R^{\alpha}\overline{u}(r)}{r(r^{2}-R^{2})^{\frac{\alpha}{2}}}dr\leq CR^{\alpha}\int_{2R}^{+\infty}\frac{\overline{u}(r)}{r^{1+\alpha}}dr=o_{2R}(1)R^{\alpha}
\end{equation}
as $R\rightarrow+\infty$. We can choose $R_{0}>0$ sufficiently large such that, $o_{2R}(1)<\frac{C}{2}$ for any $R>R_{0}$ with the same constant $C$ as in the RHS of \eqref{2-21}. Consequently, it follows from \eqref{2-21} and \eqref{2-23} that
\begin{equation}\label{2-25}
\int_{R}^{2R}\frac{R^{\alpha}\overline{u}(r)}{r\left(r^{2}-R^{2}\right)^{\frac{\alpha}{2}}}dr>C R^{\alpha}, \qquad \forall \,\, R>R_{0}.
\end{equation}
By \eqref{2-18}, we arrive at
\begin{eqnarray}\label{2-26}
  &&\int_{1}^{+\infty}\frac{1}{R^{1+\alpha}}\int_{R}^{2R}\frac{R^{\alpha}\overline{u}(r)}{r(r^{2}-R^{2})^{\frac{\alpha}{2}}}drdR
  =\int_{1}^{+\infty}\frac{\overline{u}(r)}{r}\int_{\frac{r}{2}}^{r}\frac{1}{R(r^{2}-R^{2})^{\frac{\alpha}{2}}}dRdr \\
 \nonumber &\leq&C\int_{1}^{+\infty}\frac{\overline{u}(r)}{r^{2+\frac{\alpha}{2}}}\int^{r}_{\frac{r}{2}}\frac{1}{\left(r-R\right)^{\frac{\alpha}{2}}}dRdr
 \leq C\int_{1}^{+\infty}\frac{\overline{u}(r)}{r^{1+\alpha}}dr<+\infty,
\end{eqnarray}
which is a contradiction with \eqref{2-25}. Therefore, the super poly-harmonic properties in Theorem \ref{Thm0} holds and hence Theorem \ref{Thm0} is proved.

\section{Proof of Theorem \ref{Thm1}}
In this section, we show sub poly-harmonic properties for nonnegative classical solutions to equations \eqref{PDE} with $f\leq0$, i.e. Theorem \ref{Thm1}.

Suppose that $u$ is a nonnegative classical solution to \eqref{PDE}. Let $u_1(x):=(-\Delta)^{\frac{\alpha}{2}}u(x)$ and $u_i(x):=(-\Delta)^{i-1}u_1(x)$ for $i=2,\cdots, m$. Then, from equations \eqref{PDE}, we have
\begin{equation}\label{2-1s}
\left\{{\begin{array}{l} {(-\Delta)^{\frac{\alpha}{2}} u(x)=u_1(x) \quad \text{in} \,\, \mathbb{R}^{n},} \\  {} \\ {-\Delta u_{1}(x)=u_{2}(x) \quad \text{in} \,\, \mathbb{R}^{n},} \\ \cdots\cdots \\ {-\Delta u_{m}(x)=f\leq0 \quad \text{in} \,\, \mathbb{R}^{n}.} \\ \end{array}}\right.
\end{equation}
Our aim is to prove that $u_i\leq 0$ in $\mathbb{R}^n$ for every $i=1,\cdots,m$.

First, we will prove $u_m\leq0$ by contradiction arguments. If not, then there exists $x_0\in\mathbb{R}^n$ such that $u_m(x_0)>0$. By taking spherical average w.r.t. center $x_{0}$ to all equations except the first equation in \eqref{2-1s}, we have
\begin{equation}\label{2-9s}
\left\{{\begin{array}{l} {{(-\Delta)}^{\frac{\alpha}{2}}u(x)=u_1(x) \quad \text{in} \,\, \mathbb{R}^{n},} \\  {} \\ {-\Delta\overline{u_{1}}(r)=\overline{u_{2}}(r), \quad \forall \, r\geq0,} \\ \cdots\cdots \\ {-\Delta\overline{u_{m}}(r)=\overline{f}(r)\leq0, \quad \forall \, r\geq0.} \\ \end{array}}\right.
\end{equation}
From last equation of \eqref{2-9s}, one has
\begin{equation}\label{2-4s}
-\frac{1}{r^{n-1}}\Big(r^{n-1}\overline{u_{m}}\,'(r)\Big)'\leq0, \qquad \forall \, r\geq0.
\end{equation}
Integrating both sides of \eqref{2-4s} twice gives
\begin{equation}\label{2-5s}
 \overline{u_{m}}(r)\geq\overline{u_{m}}(0)=u_{m}(x_{0})=:c_{0}>0,
\end{equation}
for any $r\geq0$. Then from the last but one equation of \eqref{2-9s} we derive
\begin{equation}\label{2-6s}
-\frac{1}{r^{n-1}}\Big(r^{n-1}\overline{u_{m-1}}\,'(r)\Big)'\geq c_0, \qquad \forall \, r\geq0.
\end{equation}
Again, by integrating both sides of \eqref{2-6s} twice, we arrive at
\begin{equation}\label{2-7s}
\overline{u_{m-1}}(r)\leq\overline{u_{m-1}}(0)-c_1 r^2, \qquad \forall \, r\geq0,
\end{equation}
where $c_1=\frac{c_0}{2n}>0$. Continuing this way, we finally obtain that
\begin{equation}\label{2-19s}
(-1)^{m-1}\overline{u_1}(r)\geq a_{m-1}r^{2(m-1)}+\cdots+a_0, \qquad \forall \, r\geq0,
\end{equation}
where $a_{m-1}>0$. Hence, we have
\begin{equation}\label{2-10s}
{(-1)}^{m-1}\overline{u_1}(r)\geq Cr^{2(m-1)},
\end{equation}
for any $r>R_0$ with $R_0$ sufficiently large. One should observe that it is very difficult to take spherical average to the first equation in \eqref{2-9s}, since the fractional Laplacian $(-\Delta)^{\frac{\alpha}{2}}$ is a nonlocal operator. Instead of taking spherical average, we will apply the Green-Poisson representation formulae for $(-\Delta)^{\frac{\alpha}{2}}$ to the first equation of \eqref{2-9s} to overcome this difficulty. From the first equation in \eqref{2-9s}, we conclude that, for arbitrary $R>0$,
\begin{equation}\label{2-20s}
  u(x)=\int_{B_{R}(x_{0})}G^\alpha_R(x,y)u_{1}(y)dy+\int_{|y-x_{0}|>R}P^{\alpha}_{R}(x,y)u(y)dy, \qquad \forall \,\, x\in B_{R}(x_{0}),
\end{equation}
where $G^\alpha_R$ is the Green's function for $(-\Delta)^{\frac{\alpha}{2}}$ with $0<\alpha<2$ on $B_R(x_{0})$ and $P^{\alpha}_{R}(x,y)$ is the Poisson kernel for $(-\Delta)^{\frac{\alpha}{2}}$ in $B_{R}(x_{0})$. Taking $x=x_{0}$ in \eqref{2-20s} gives
\begin{align}\label{2-11s}
u(x_{0})&=\int_{B_R(x_{0})}\frac{C_{n,\alpha}}{|y-x_{0}|^{n-\alpha}}\left(\int_0^{\frac{R^2}{|y-x_{0}|^2}-1}\frac{b^{\frac{\alpha}{2}-1}}{(1+b)^{\frac{n}{2}}}db\right)u_1(y)dy \\
\nonumber &\quad +C'_{n,\alpha}\int_{|y-x_{0}|>R}\frac{R^\alpha}{{(|y-x_{0}|^2-R^2)}^{\frac{\alpha}{2}}}\frac{u(y)}{|y-x_{0}|^n}dy\\
&=\int_{0}^{R}\frac{\widetilde{C}_{n,\alpha}}{r^{1-\alpha}}\left(\int_0^{\frac{R^2}{r^2}-1} \frac{b^{\frac{\alpha}{2}-1}}{(1+b)^{\frac{n}{2}}}db\right)\overline{u_1}(r)dr +\overline{C}_{n,\alpha}\int_{R}^{+\infty}\frac{R^\alpha\bar{u}(r)}{r{(r^2-R^2)}^{\frac{\alpha}{2}}}dr. \nonumber
\end{align}
One can easily observe that for $0<r\leq\frac{R}{2}$, $\frac{R^2}{r^2}-1\geq3$ and hence $\int_{0}^{\frac{R^2}{r^2}-1}\frac{b^{\frac{\alpha}{2}-1}}{(1+b)^{\frac{n}{2}}}db\geq\int_{0}^{3}\frac{b^{\frac{\alpha}{2}-1}}{(1+b)^{\frac{n}{2}}}db=:C_1>0$. For $\frac{R}{2}<r<R$, one has $0<\frac{R^2}{r^2}-1<3$, thus $\int_{0}^{\frac{R^2}{r^2}-1}\frac{b^{\frac{\alpha}{2}-1}}{(1+b)^{\frac{n}{2}}}db>\int_{0}^{\frac{R^2}{r^2}-1}\frac{b^{\frac{\alpha}{2}-1}}{2^n}db
=:C_2{\left(\frac{R^2-r^2}{r^2}\right)}^{\frac{\alpha}{2}}$. Thus, we conclude that
\begin{equation}\label{2-12s}
\int_0^{\frac{R^2}{r^2}-1}\frac{b^{\frac{\alpha}{2}-1}}{(1+b)^{\frac{n}{2}}}db\geq C_1\chi_{0<r\leq\frac{R}{2}}+C_2\chi_{\frac{R}{2}<r<R}{\left(\frac{R^2-r^2}{r^2}\right)}^{\frac{\alpha}{2}}
\end{equation}
for any $0<r<R$. Then, from \eqref{2-10s}, \eqref{2-11s} and \eqref{2-12s}, we derive that, for $R>2R_0$,
\begin{align}\label{2-13s}
&\quad{(-1)}^{m}\int_{R}^{+\infty}\frac{R^\alpha\bar{u}(r)}{r{(r^2-R^2)}^{\frac{\alpha}{2}}} dr \\
&=C\int_0^R \frac{1}{r^{1-\alpha}}\left(\int_0^{\frac{R^2}{r^2}-1}\frac{b^{\frac{\alpha}{2}-1}}{(1+b)^{\frac{n}{2}}}db\right) {(-1)}^{m-1}\overline{u_1}(r)dr+{(-1)}^{m}u(x_{0})\nonumber\\
&\geq C\int_{R_0}^R \frac{1}{r^{1-\alpha}}\left[C_1 \chi_{0<r<\frac{R}{2}}
+C_2\chi_{\frac{R}{2}<r<R}{\left(\frac{R^2-r^2}{r^2}\right)}^{\frac{\alpha}{2}}\right] r^{2(m-1)}dr\nonumber \\
&\,\,\,\,\,\,\,+C\int_{0}^{R_0}\frac{1}{r^{1-\alpha}}\left(\int_{0}^{\frac{R^2}{r^2}-1}\frac{b^{\frac{\alpha}{2}-1}}{(1+b)^{\frac{n}{2}}}db\right)
(-1)^{m-1}\overline{u_1}(r)dr+{(-1)}^{m}u(x_{0}) \nonumber\\
&\geq CR^{2m-2+\alpha}-C_0+{(-1)}^{m}u(x_{0})\nonumber.
\end{align}
It is easy to see that if $m$ is odd, \eqref{2-13s} implies that $\int_{R}^{+\infty}\frac{R^\alpha\bar{u}(r)}{r{(r^2-R^2)}^{\frac{\alpha}{2}}}dr<0$ for $R$ sufficiently large, which contradicts with $u\geq0$. Therefore, we only need to consider the case that $m$ is even. In such cases, \eqref{2-13s} implies
\begin{equation}\label{2-14s}
\int_{R}^{+\infty}\frac{R^\alpha\bar{u}(r)}{r{(r^2-R^2)}^{\frac{\alpha}{2}}}dr\geq CR^{2m-2+\alpha}
\end{equation}
for $R$ sufficiently large. Since $u\in \mathcal{L}_{\alpha}(\mathbb{R}^{n})$, we have
\begin{equation} \label{2-15s}
\int_{|x-x_{0}|>1}\frac{u(x)}{|x-x_{0}|^{n+\alpha}}dx=C\int_1^{+\infty} \frac{\bar{u}(r)}{r^{1+\alpha}}dr<+\infty.
\end{equation}
Then, by \eqref{2-15s} and the fact that $2m-2\geq0$, for $R$ sufficiently large, we have
\begin{align}\label{2-16s}
\int_{R}^{2R}\frac{R^\alpha\bar{u}(r)}{r{(r^2-R^2)}^{\frac{\alpha}{2}}}dr
&\geq CR^{2m-2+\alpha}-\int_{2R}^{+\infty}\frac{R^\alpha\bar{u}(r)}{r{(r^2-R^2)}^{\frac{\alpha}{2}}} dr\\
&\geq CR^{2m-2+\alpha}-C'R^\alpha\int_{2R}^{+\infty} \frac{\bar{u}(r)}{r^{1+\alpha}}dr \nonumber\\
&\geq CR^{2m-2+\alpha}-o_R(1)R^{\alpha}\nonumber\\
&\geq CR^{2m-2+\alpha}.\nonumber
\end{align}
On the one hand, by \eqref{2-15s}, we have
\begin{align}\label{2-17s}
\int_1^{+\infty}\frac{1}{R^{1+\alpha}}\int_R^{2R}\frac{R^{\alpha}\bar{u}(r)}{r{(r^2-R^2)}^{\frac{\alpha}{2}}}drdR
&=\int_1^{+\infty}\frac{\bar{u}}{r}\int_{\frac{r}{2}}^{r}\frac{1}{R{(r^2-R^2)}^{\frac{\alpha}{2}}}dRdr\\
&\leq C\int_1^{+\infty}\frac{\bar{u}(r)}{r^{1+\alpha}}dr<+\infty.\nonumber
\end{align}
On the other hand, by \eqref{2-16s}, we derive
\begin{equation}\label{2-18s}
\int_1^{+\infty}\frac{1}{R^{1+\alpha}}\int_R^{2R}\frac{R^{\alpha}\bar{u}(r)}{r{(r^2-R^2)}^{\frac{\alpha}{2}}} drdR\geq\int_{N}^{+\infty}\frac{1}{R}dR=+\infty,
\end{equation}
where $N$ is sufficiently large such that \eqref{2-16s} holds for any $R>N$. Combining \eqref{2-17s} with \eqref{2-18s}, we get a contradiction. Hence, we must have $u_m\leq 0$ in $\mathbb{R}^n$. One should observe that, in the proof of $u_{m}\leq0$, we have mainly used the property $-\Delta u_m\leq 0$. Therefore, through a similar argument as above, one can prove that $u_{m-1}\leq 0$. Continuing this way, we obtain that $u_i=(-\Delta)^{i-1+\frac{\alpha}{2}}u\leq0$ for every $i=1,2,\cdots,m$. This completes the proof of Theorem \ref{Thm1}.

\section{Proof of Theorem \ref{Thm2}}
In this section, we will prove Theorem \ref{Thm2} by way of contradiction  and the method of scaling spheres developed by Dai and Qin \cite{DQ2} (see also \cite{DQ3,DQ4,DQZ}). For more related literature on the method of moving planes (spheres), we refer to \cite{CD,CDQ,CF,CFY,CGS,CL0,CL3,CL1,CLL,CLO,DFHQW,DL,DPQ,DQ1,DQ5,GNN,LiC,Lin,Li,LZ,WX,Xu} and the references therein.

Now suppose, on the contrary, that $u\geq0$ satisfies integral equations \eqref{IE} but $u$ is not identically zero, then there exists a ponit $\bar{x}\in\mathbb{R}^{n}$ such that $u(\bar{x})>0$. It follows from \eqref{IE} immediately that
\begin{equation}\label{3-0-2}
  u(x)>0, \,\,\,\,\,\,\, \forall \,\, x\in\mathbb{R}^{n},
\end{equation}
i.e., $u$ is actually a positive solution in $\mathbb{R}^{n}$. Moreover, there exists a constant $C>0$, such that the solution $u$ satisfies the following lower bound:
\begin{equation}\label{3-1}
  u(x)\geq\frac{C}{|x|^{n-2m-\alpha}} \,\,\,\,\,\,\,\,\,\,\,\,\, \text{for} \,\,\, |x|\geq1.
\end{equation}
Indeed, since $u>0$ satisfies the integral equation \eqref{IE}, we can infer that
\begin{eqnarray}\label{3-2}
  u(x)&\geq& C_{n,m,\alpha}\int_{|y|\leq\frac{1}{2}}\frac{|y|^{a}}{|x-y|^{n-2m-\alpha}}u^{p}(y)dy\\
 \nonumber &\geq& \frac{C}{|x|^{n-2m-\alpha}}\int_{|y|\leq\frac{1}{2}}|y|^{a}u^{p}(y)dy=:\frac{C}{|x|^{n-2m-\alpha}}
\end{eqnarray}
for all $|x|\geq1$.

Next, we will apply the method of scaling spheres to show the following lower bound estimates for positive solution $u$, which contradict with the integral equations \eqref{IE} for $0<p<\frac{n+2m+\alpha+2a}{n-2m-\alpha}$.
\begin{thm}\label{lower1}
Assume $m\geq1$, $0<\alpha<2$, $2m+\alpha<n$, $0\leq a<+\infty$, $0<p<\frac{n+2m+\alpha+2a}{n-2m-\alpha}$. Suppose $u\in C(\mathbb{R}^{n})$ is a positive solution to \eqref{IE}, then it satisfies the following lower bound estimates: for $|x|\geq1$,
\begin{equation}\label{lb1-1}
  u(x)\geq C_{\kappa}|x|^{\kappa} \quad\quad \forall \, \kappa<\frac{2m+\alpha+a}{1-p}, \quad\quad \text{if} \,\,\,\, 0<p<1;
\end{equation}
\begin{equation}\label{lb2-1}
  u(x)\geq C_{\kappa}|x|^{\kappa} \quad\quad \forall \, \kappa<+\infty, \quad\quad \text{if} \,\,\,\, 1\leq p<\frac{n+2m+\alpha+2a}{n-2m-\alpha}.
\end{equation}
\end{thm}
\begin{proof}
Given any $\lambda>0$, we first define the Kelvin transform of a function $u:\,\mathbb{R}^{n}\rightarrow\mathbb{R}$ centered at $0$ by
\begin{equation}\label{Kelvin1}
  u_{\lambda}(x)=\left(\frac{\lambda}{|x|}\right)^{n-2m-\alpha}u\left(\frac{\lambda^{2}x}{|x|^{2}}\right)
\end{equation}
for arbitrary $x\in\mathbb{R}^{n}\setminus\{0\}$. It's obvious that the Kelvin transform $u_{\lambda}$ may have singularity at $0$ and $\lim_{|x|\rightarrow\infty}|x|^{n-2m-\alpha}u_{\lambda}(x)=\lambda^{n-2m-\alpha}u(0)>0$. By \eqref{Kelvin1}, one can infer from the regularity assumptions on $u$ that $u_{\lambda}\in C(\mathbb{R}^{n}\setminus\{0\})$.

Next, we will carry out the process of scaling spheres with respect to the origin $0\in\mathbb{R}^{n}$.

To this end, let $\lambda>0$ be an arbitrary positive real number and let
\begin{equation}\label{3-1-0}
  \omega^{\lambda}(x):=u_{\lambda}(x)-u(x)
\end{equation}
for any $x\in B_{\lambda}(0)\setminus\{0\}$. We will first show that, for $\lambda>0$ sufficiently small,
\begin{equation}\label{3-7}
  \omega^{\lambda}(x)\geq0, \,\,\,\,\,\, \forall \,\, x\in B_{\lambda}(0)\setminus\{0\}.
\end{equation}
Then, we start dilating the sphere $S_{\lambda}$ from a place near the origin $0$ outward as long as \eqref{3-7} holds, until its limiting position $\lambda=+\infty$ and derive lower bound estimates on $u$. Therefore, the scaling sphere process can be divided into two steps.

\emph{Step 1. Start dilating the sphere from near $\lambda=0$.} Define
\begin{equation}\label{3-8}
  B^{-}_{\lambda}:=\{x\in B_{\lambda}(0)\setminus\{0\} \, | \, \omega^{\lambda}(x)<0\}.
\end{equation}
We will show that, for $\lambda>0$ sufficiently small,
\begin{equation}\label{3-9}
  B^{-}_{\lambda}=\emptyset.
\end{equation}

Since $u\in C(\mathbb{R}^{n})$ is a positive solution to integral equations \eqref{IE}, through direct calculations, we get
\begin{equation}\label{3-10}
  u(x)=C\int_{B_{\lambda}(0)}\frac{|y|^{a}}{|x-y|^{n-2m-\alpha}}u^{p}(y)dy
  +C\int_{B_{\lambda}(0)}\frac{|y|^{a}}{\left|\frac{|y|}{\lambda}x-\frac{\lambda}{|y|}y\right|^{n-2m-\alpha}}
  \left(\frac{\lambda}{|y|}\right)^{\tau}u_{\lambda}^{p}(y)dy
\end{equation}
for any $x\in\mathbb{R}^{n}$, where $\tau:=n+2m+\alpha+2a-p(n-2m-\alpha)>0$. Direct calculations deduce that $u_{\lambda}$ satisfies the following integral equation
\begin{equation}\label{3-11}
  u_{\lambda}(x)=C\int_{\mathbb{R}^{n}}\frac{|y|^{a}}{|x-y|^{n-2m-\alpha}}\left(\frac{\lambda}{|y|}\right)^{\tau}u_{\lambda}^{p}(y)dy
\end{equation}
for any $x\in\mathbb{R}^{n}\setminus\{0\}$, and hence, it follows immediately that
\begin{equation}\label{3-12}
  u_{\lambda}(x)=C\int_{B_{\lambda}(0)}\frac{|y|^{a}}{\left|\frac{|y|}{\lambda}x-\frac{\lambda}{|y|}y\right|^{n-2m-\alpha}}u^{p}(y)dy
  +C\int_{B_{\lambda}(0)}\frac{|y|^{a}}{|x-y|^{n-2m-\alpha}}
  \left(\frac{\lambda}{|y|}\right)^{\tau}u_{\lambda}^{p}(y)dy.
\end{equation}
From the integral equations \eqref{3-10} and \eqref{3-12}, one can derive that, for any $x\in B_{\lambda}^{-}$,
\begin{eqnarray}\label{3-13}
  &&0>\omega^{\lambda}(x)=u_{\lambda}(x)-u(x) \\
 \nonumber &=&C\int_{B_{\lambda}(0)}\Bigg(\frac{|y|^{a}}{|x-y|^{n-2m-\alpha}}-\frac{|y|^{a}}{\left|\frac{|y|}{\lambda}x-\frac{\lambda}{|y|}y\right|^{n-2m-\alpha}}\Bigg) \left(\left(\frac{\lambda}{|y|}\right)^{\tau}u_{\lambda}^{p}(y)-u^{p}(y)\right)dy\\
\nonumber &>& C\int_{B_{\lambda}^{-}}\Bigg(\frac{|y|^{a}}{|x-y|^{n-2m-\alpha}}-\frac{|y|^{a}}{\left|\frac{|y|}{\lambda}x-\frac{\lambda}{|y|}y\right|^{n-2m-\alpha}}\Bigg)
\max\left\{u^{p-1}(y),u_{\lambda}^{p-1}(y)\right\}\omega^{\lambda}(y)dy\\
\nonumber &\geq& C\int_{B_{\lambda}^{-}}\frac{|y|^{a}}{|x-y|^{n-2m-\alpha}}\max\left\{u^{p-1}(y),u_{\lambda}^{p-1}(y)\right\}\omega^{\lambda}(y)dy.
\end{eqnarray}

By Hardy-Littlewood-Sobolev inequality and \eqref{3-13}, we have, for any $\frac{n}{n-2m-\alpha}<q<\infty$,
\begin{eqnarray}\label{3-14}
  \|\omega^{\lambda}\|_{L^{q}(B^{-}_{\lambda})}&\leq& C\left\||x|^{a}\max\left\{u^{p-1},u_{\lambda}^{p-1}\right\}\omega^{\lambda}\right\|_{L^{\frac{nq}{n+(2m+\alpha)q}}(B^{-}_{\lambda})}\\
  \nonumber &\leq& C\left\||x|^{a}\max\left\{u^{p-1},u_{\lambda}^{p-1}\right\}\right\|_{L^{\frac{n}{2m+\alpha}}(B^{-}_{\lambda})}\|\omega^{\lambda}\|_{L^{q}(B^{-}_{\lambda})}.
\end{eqnarray}
Since \eqref{3-2} implies that
\begin{equation}\label{inf}
  \inf_{x\in B_{\lambda}(0)\setminus\{0\}}u_{\lambda}(x)\geq C
\end{equation}
for any $\lambda\leq1$, there exists a $\epsilon_{0}>0$ small enough, such that
\begin{equation}\label{3-15}
  C\left\||x|^{a}\max\left\{u^{p-1},u_{\lambda}^{p-1}\right\}\right\|_{L^{\frac{n}{2m+\alpha}}(B^{-}_{\lambda})}\leq\frac{1}{2}
\end{equation}
for all $0<\lambda\leq\epsilon_{0}$. Thus \eqref{3-14} implies
\begin{equation}\label{3-16}
  \|\omega^{\lambda}\|_{L^{q}(B^{-}_{\lambda})}=0, \qquad \forall \,\, 0<\lambda\leq\epsilon_{0},
\end{equation}
which means $B^{-}_{\lambda}=\emptyset$. Consequently for all $0<\lambda\leq\epsilon_{0}$,
\begin{equation}\label{3-17}
  \omega^{\lambda}(x)\geq0, \,\,\,\,\,\,\, \forall \, x\in B_{\lambda}(0)\setminus\{0\},
\end{equation}
which completes Step 1.

\emph{Step 2. Dilate the sphere $S_{\lambda}$ outward until $\lambda=+\infty$ to derive lower bound estimates on $u$.} Step 1 provides us a start point to dilate the sphere $S_{\lambda}$ from place near $\lambda=0$. Now we dilate the sphere $S_{\lambda}$ outward as long as \eqref{3-7} holds. Let
\begin{equation}\label{3-18}
  \lambda_{0}:=\sup\{\lambda>0\,|\, \omega^{\mu}\geq0 \,\, in \,\, B_{\mu}(0)\setminus\{0\}, \,\, \forall \, 0<\mu\leq\lambda\}\in(0,+\infty],
\end{equation}
and hence, one has
\begin{equation}\label{3-19}
  \omega^{\lambda_{0}}(x)\geq0, \quad\quad \forall \,\, x\in B_{\lambda_{0}}(0)\setminus\{0\}.
\end{equation}
In what follows, we will prove $\lambda_{0}=+\infty$ by contradiction arguments.

Suppose on contrary that $0<\lambda_{0}<+\infty$. In order to get a contradiction, we will first show that
\begin{equation}\label{3-21}
  \omega^{\lambda_{0}}(x)>0, \,\,\,\,\,\, \forall \, x\in B_{\lambda_{0}}(0)\setminus\{0\}.
\end{equation}
Then, we will obtain a contradiction with \eqref{3-18} via showing that the sphere $S_{\lambda}$ can be dilated outward a little bit further. More precisely, there exists a $\varepsilon>0$ small enough such that $\omega^{\lambda}\geq0$ in $B_{\lambda}(0)\setminus\{0\}$ for all $\lambda\in[\lambda_{0},\lambda_{0}+\varepsilon]$.

Now we start to prove \eqref{3-21}. Indeed, if we suppose that
\begin{equation}\label{3-20}
  \omega^{\lambda_{0}}(x)\equiv0, \,\,\,\,\,\,\forall \, x\in B_{\lambda_{0}}(0)\setminus\{0\},
\end{equation}
then by the second equality in \eqref{3-13} and \eqref{3-20}, we arrive at
\begin{eqnarray}\label{3-34}
 && 0=\omega^{\lambda_{0}}(x)=u_{\lambda_{0}}(x)-u(x)\\
 \nonumber &=&C\int_{B_{\lambda_{0}}(0)}\Bigg(\frac{|y|^{a}}{|x-y|^{n-2m-\alpha}}-\frac{|y|^{a}}{\left|\frac{|y|}{\lambda_{0}}x-\frac{\lambda_{0}}{|y|}y\right|^{n-2m-\alpha}}\Bigg) \left(\left(\frac{\lambda_{0}}{|y|}\right)^{\tau}-1\right)u^{p}(y)dy>0
\end{eqnarray}
for any $x\in B_{\lambda_{0}}(0)\setminus\{0\}$, which is absurd. Thus there exists a point $x^{0}\in B_{\lambda_{0}}(0)\setminus\{0\}$ such that $\omega^{\lambda_{0}}(x^{0})>0$, which implies that by continuity, there exists a small $\delta>0$ and a constant $c_{0}>0$ such that
\begin{equation}\label{3-22}
B_{\delta}(x^{0})\subset B_{\lambda_{0}}(0)\setminus\{0\} \,\,\,\,\,\, \text{and} \,\,\,\,\,\,
\omega^{\lambda_{0}}(x)\geq c_{0}>0, \,\,\,\,\,\,\,\, \forall \, x\in B_{\delta}(x^{0}).
\end{equation}
From \eqref{3-22} and the integral equations \eqref{3-10} and \eqref{3-12}, one can derive that, for any $x\in B_{\lambda_{0}}(0)\setminus\{0\}$,
\begin{eqnarray}\label{3-23}
  &&\omega^{\lambda_{0}}(x)=u_{\lambda_{0}}(x)-u(x) \\
 \nonumber &=&C\int_{B_{\lambda_{0}}(0)}\Bigg(\frac{|y|^{a}}{|x-y|^{n-2m-\alpha}}-\frac{|y|^{a}}{\left|\frac{|y|}{\lambda_{0}}x-\frac{\lambda_{0}}{|y|}y\right|^{n-2m-\alpha}}\Bigg) \left(\left(\frac{\lambda_{0}}{|y|}\right)^{\tau}u_{\lambda_{0}}^{p}(y)-u^{p}(y)\right)dy\\
 \nonumber &>& C\int_{B_{\lambda_{0}}(0)}\Bigg(\frac{|y|^{a}}{|x-y|^{n-2m-\alpha}}-\frac{|y|^{a}}{\left|\frac{|y|}{\lambda_{0}}x-\frac{\lambda_{0}}{|y|}y\right|^{n-2m-\alpha}}\Bigg) \left(u_{\lambda_{0}}^{p}(y)-u^{p}(y)\right)dy\\
 \nonumber &\geq& C\int_{B_{\lambda_{0}}(0)}\Bigg(\frac{|y|^{a}}{|x-y|^{n-2m-\alpha}}-\frac{|y|^{a}}{\left|\frac{|y|}{\lambda_{0}}x-\frac{\lambda_{0}}{|y|}y\right|^{n-2m-\alpha}}\Bigg)
\min\left\{u^{p-1}(y),u_{\lambda_{0}}^{p-1}(y)\right\}\omega^{\lambda_{0}}(y)dy\\
\nonumber &\geq& C\int_{B_{\delta}(x^{0})}\Bigg(\frac{|y|^{a}}{|x-y|^{n-2m-\alpha}}-\frac{|y|^{a}}{\left|\frac{|y|}{\lambda_{0}}x-\frac{\lambda_{0}}{|y|}y\right|^{n-2m-\alpha}}\Bigg)
\min\left\{u^{p-1}(y),u_{\lambda_{0}}^{p-1}(y)\right\}\omega^{\lambda_{0}}(y)dy>0,
\end{eqnarray}
and thus we arrive at \eqref{3-21}. Furthermore, \eqref{3-23} also implies that there exists a $0<\eta<\lambda_{0}$ small enough such that, for any $x\in \overline{B_{\eta}(0)}\setminus\{0\}$,
\begin{equation}\label{3-24}
  \omega^{\lambda_{0}}(x)>c_{4}+C\int_{B_{\frac{\delta}{2}}(x^{0})}c_{3}^{a}\,c_{2}\,c_{1}^{p-1}c_{0} \, dy=:\widetilde{c}_{0}>0.
\end{equation}

Next, we will show that the sphere $S_{\lambda}$ can be dilated outward a little bit further and hence obtain a contradiction with the definition \eqref{3-18} of $\lambda_{0}$.

To this end, we fixed $0<r_{0}<\frac{1}{2}\lambda_{0}$ small enough, such that
\begin{equation}\label{3-25}
  C\left\||x|^{a}\max\left\{u^{p-1},u_{\lambda}^{p-1}\right\}\right\|_{L^{\frac{n}{2m+\alpha}}(A_{\lambda_{0}+r_{0},2r_{0}})}\leq\frac{1}{2}
\end{equation}
for any $\lambda\in[\lambda_{0},\lambda_{0}+r_{0}]$, where the constant $C$ is the same as in \eqref{3-14} and the narrow region
\begin{equation}\label{3-26}
  A_{\lambda_{0}+r_{0},2r_{0}}:=\left\{x\in B_{\lambda_{0}+r_{0}}(0)\,|\,|x|>\lambda_{0}-r_{0}\right\}.
\end{equation}
By \eqref{3-13}, one can easily verify that inequality as \eqref{3-14} (with the same constant $C$) also holds for any $\lambda\in[\lambda_{0},\lambda_{0}+r_{0}]$, that is, for any $\frac{n}{n-2m-\alpha}<q<\infty$,
\begin{equation}\label{3-27}
  \|\omega^{\lambda}\|_{L^{q}(B^{-}_{\lambda})}\leq C\left\||x|^{a}\max\left\{u^{p-1},u_{\lambda}^{p-1}\right\}\right\|_{L^{\frac{n}{2m+\alpha}}(B^{-}_{\lambda})}\|\omega^{\lambda}\|_{L^{q}(B^{-}_{\lambda})}.
\end{equation}
From \eqref{3-21} and \eqref{3-24}, we can infer that
\begin{equation}\label{3-28}
  m_{0}:=\inf_{x\in\overline{B_{\lambda_{0}-r_{0}}(0)}\setminus\{0\}}\omega^{\lambda_{0}}(x)>0.
\end{equation}
Since $u$ is uniformly continuous on arbitrary compact set $K\subset\mathbb{R}^{n}$ (say, $K=\overline{B_{4\lambda_{0}}(0)}$), we can deduce from \eqref{3-28} that, there exists a $0<\varepsilon_{1}<r_{0}$ sufficiently small, such that, for any $\lambda\in[\lambda_{0},\lambda_{0}+\varepsilon_{1}]$,
\begin{equation}\label{3-29}
  \omega^{\lambda}(x)\geq\frac{m_{0}}{2}>0, \,\,\,\,\,\, \forall \, x\in\overline{B_{\lambda_{0}-r_{0}}(0)}\setminus\{0\}.
\end{equation}
In order to prove \eqref{3-29}, one should observe that \eqref{3-28} is equivalent to
\begin{equation}\label{3-50}
  |x|^{n-2m-\alpha}u(x)-\lambda_{0}^{n-2m-\alpha}u(x^{\lambda_{0}})\geq m_{0}\lambda_{0}^{n-2m-\alpha}, \,\,\,\,\,\,\,\,\, \forall \, |x|\geq\frac{\lambda_{0}^{2}}{\lambda_{0}-r_{0}}.
\end{equation}
Since $u$ is uniformly continuous on $\overline{B_{4\lambda_{0}}(0)}$, we infer from \eqref{3-50} that there exists a $0<\varepsilon_{1}<r_{0}$ sufficiently small, such that, for any $\lambda\in[\lambda_{0},\lambda_{0}+\varepsilon_{1}]$,
\begin{equation}\label{2-43}
  |x|^{n-2m-\alpha}u(x)-\lambda^{n-2m-\alpha}u(x^{\lambda})\geq \frac{m_{0}}{2}\lambda^{n-2m-\alpha}, \,\,\,\,\,\,\,\,\, \forall \, |x|\geq\frac{\lambda^{2}}{\lambda_{0}-r_{0}},
\end{equation}
which is equivalent to \eqref{3-29}, hence we have proved \eqref{3-29}.

By \eqref{3-29}, we know that for any $\lambda\in[\lambda_{0},\lambda_{0}+\varepsilon_{1}]$,
\begin{equation}\label{3-30}
  B_{\lambda}^{-}\subset A_{\lambda_{0}+r_{0},2r_{0}},
\end{equation}
and hence, estimates \eqref{3-25} and \eqref{3-27} yields
\begin{equation}\label{3-31}
  \|\omega^{\lambda}\|_{L^{q}(B^{-}_{\lambda})}=0.
\end{equation}
Therefore, for any $\lambda\in[\lambda_{0},\lambda_{0}+\varepsilon_{1}]$, we deduce from \eqref{3-31} that, $B^{-}_{\lambda}=\emptyset$, that is,
\begin{equation}\label{3-32}
  \omega^{\lambda}(x)\geq0, \,\,\,\,\,\,\, \forall \,\, x\in B_{\lambda}(0)\setminus\{0\},
\end{equation}
which contradicts with the definition \eqref{3-18} of $\lambda_{0}$. Thus we must have $\lambda_{0}=+\infty$, that is,
\begin{equation}\label{3-35}
  u(x)\geq\left(\frac{\lambda}{|x|}\right)^{n-2m-\alpha}u\left(\frac{\lambda^{2}x}{|x|^{2}}\right), \quad\quad \forall \,\, |x|\geq\lambda, \quad \forall \,\, 0<\lambda<+\infty.
\end{equation}
For arbitrary $|x|\geq1$, let $\lambda:=\sqrt{|x|}$, then \eqref{3-35} yields that
\begin{equation}\label{3-36}
  u(x)\geq\frac{1}{|x|^{\frac{n-2m-\alpha}{2}}}u\left(\frac{x}{|x|}\right),
\end{equation}
and hence, we arrive at the following lower bound estimate:
\begin{equation}\label{3-37}
  u(x)\geq\left(\min_{x\in S_{1}}u(x)\right)\frac{1}{|x|^{\frac{n-2m-\alpha}{2}}}:=\frac{C_{0}}{|x|^{\frac{n-2m-\alpha}{2}}}, \quad\quad \forall \,\, |x|\geq1.
\end{equation}

The lower bound estimate \eqref{3-37} can be improved remarkably for $0<p<\frac{n+2m+\alpha+2a}{n-2m-\alpha}$ using the ``Bootstrap" iteration technique and the integral equations \eqref{IE}.

In fact, let $\mu_{0}:=\frac{n-2m-\alpha}{2}$, we infer from the integral equations \eqref{IE} and \eqref{3-37} that, for $|x|\geq1$,
\begin{eqnarray}\label{3-38}
  u(x)&\geq&C\int_{2|x|\leq|y|\leq4|x|}\frac{1}{|x-y|^{n-2m-\alpha}|y|^{p\mu_{0}-a}}dy \\
  \nonumber &\geq&\frac{C}{|x|^{n-2m-\alpha}}\int_{2|x|\leq|y|\leq4|x|}\frac{1}{|y|^{p\mu_{0}-a}}dy \\
  \nonumber &\geq&\frac{C}{|x|^{n-2m-\alpha}}\int^{4|x|}_{2|x|}r^{n-1-p\mu_{0}+a}dr \\
  \nonumber &\geq&\frac{C_{1}}{|x|^{p\mu_{0}-(a+2m+\alpha)}}.
\end{eqnarray}
Let $\mu_{1}:=p\mu_{0}-(a+2m+\alpha)$. Due to $0<p<\frac{n+2m+\alpha+2a}{n-2m-\alpha}$, our important observation is
\begin{equation}\label{3-39}
  \mu_{1}:=p\mu_{0}-(a+2m+\alpha)<\mu_{0}.
\end{equation}
Thus we have obtained a better lower bound estimate than \eqref{3-37} after one iteration, that is,
\begin{equation}\label{3-40}
  u(x)\geq\frac{C_{1}}{|x|^{\mu_{1}}}, \quad\quad \forall \,\, |x|\geq1.
\end{equation}

For $k=0,1,2,\cdots$, define
\begin{equation}\label{3-41}
  \mu_{k+1}:=p\mu_{k}-(a+2m+\alpha).
\end{equation}
Since $0<p<\frac{n+2m+\alpha+2a}{n-2m-\alpha}$, it is easy to see that the sequence $\{\mu_{k}\}$ is monotone decreasing with respect to $k$. Repeating the above iteration process involving the integral equation \eqref{IE}, we have the following lower bound estimates for every $k=0,1,2,\cdots$,
\begin{equation}\label{3-42}
  u(x)\geq\frac{C_{k}}{|x|^{\mu_{k}}}, \quad\quad \forall \,\, |x|\geq1.
\end{equation}
Now Theorem \ref{lower1} follows easily from the obvious properties that as $k\rightarrow+\infty$,
\begin{equation}\label{3-43}
 \mu_{k}\rightarrow-\frac{a+2m+\alpha}{1-p} \quad \text{if} \,\, 0<p<1;
  \quad\quad \mu_{k}\rightarrow-\infty \quad \text{if} \,\, 1\leq p<\frac{n+2m+\alpha+2a}{n-2m-\alpha}.
\end{equation}
This finishes our proof of Theorem \ref{lower1}.
\end{proof}

We have proved that a nontrivial nonnegative solution $u$ to integral equations \eqref{IE} is actually a positive solution. For $0<p<\frac{n+2m+\alpha+2a}{n-2m-\alpha}$, the lower bound estimates in Theorem \ref{lower1} contradicts with the following integrability
\begin{equation}\label{3-44}
  C\int_{\mathbb{R}^{n}}\frac{u^{p}(x)}{|x|^{n-2m-\alpha-a}}dx=u(0)<+\infty
\end{equation}
indicated by  integral equations \eqref{IE}. Therefore, $u\equiv0$ in $\mathbb{R}^{n}$, that is, the unique nonnegative solution to IEs \eqref{IE} is $u\equiv0$ in $\mathbb{R}^{n}$. The proof of Theorem \ref{Thm2} is therefore completed.

\section{Proof of Theorem \ref{Thm3}}
In this section, using Theorem \ref{Thm0} and the arguments from Chen, Dai and Qin \cite{CDQ}, we will prove the Liouville properties in Theorem \ref{Thm3} in both critical order cases $m+\frac{\alpha}{2}=\frac{n}{2}$ and super-critical order cases $m+\frac{\alpha}{2}>\frac{n}{2}$.

We will prove Theorem \ref{Thm3} by using contradiction arguments. Suppose on the contrary that $u\geq0$ satisfies equation \eqref{PDE} but $u$ is not identically zero, then there exists a point $\bar{x}\in\mathbb{R}^{n}$ such that $u(\bar{x})>0$. By Theorem \ref{Thm0}, we can deduce from $(-\Delta)^{\frac{\alpha}{2}}u\geq0$, $u\geq0$, $u(\bar{x})>0$ that
\begin{equation}\label{2-50}
  u(x)>0, \,\,\,\,\,\,\, \forall \,\, x\in\mathbb{R}^{n}.
\end{equation}
Suppose not, then there exists a point $\widetilde{x}\in \mathbb{R}^{n}$ such that $u(\widetilde{x})=0$, and hence we have
\begin{equation}\label{2-52}
  (-\Delta)^{\frac{\alpha}{2}}u(\widetilde{x})=C_{n,\alpha} \, P.V.\int_{\mathbb{R}^n}\frac{-u(y)}{|\widetilde{x}-y|^{n+\alpha}}dy<0,
\end{equation}
which is absurd. Moreover, by maximum principle and induction, we can also infer further from $(-\Delta)^{i+\frac{\alpha}{2}} u\geq0$ ($i=0,\cdots,m-1$), $u>0$, the assumptions on $f$ and equation \eqref{PDE} that
\begin{equation}\label{2-51}
  (-\Delta)^{i+\frac{\alpha}{2}}u(x)>0, \,\,\,\,\,\,\,\, \forall \,\, i=0,\cdots,m-1, \,\,\,\, \forall \,\, x\in\mathbb{R}^{n}.
\end{equation}

Since $m+\frac{\alpha}{2}\geq\frac{n}{2}$, it follows immediately that either $m=\frac{n-1}{2}$ with $n$ odd or $m\geq\lceil\frac{n}{2}\rceil$, where $\lceil x\rceil$ denotes the least integer not less than $x$.

In the following, we will try to obtain contradictions by discussing the two different cases $m=\frac{n-1}{2}$ with $n$ odd and $m\geq\lceil\frac{n}{2}\rceil$ separately.

\emph{Case i): $m=\frac{n-1}{2}$ and $n$ is odd.} Since $m+\frac{\alpha}{2}\geq\frac{n}{2}$, we have $1\leq\alpha<2$. Now we will first show that $(-\Delta)^{m-1+\frac{\alpha}{2}}u$ satisfies the following integral equation
\begin{equation}\label{2c1}
  (-\Delta)^{m-1+\frac{\alpha}{2}}u(x)=\int_{\mathbb{R}^{n}}\frac{R_{2,n}}{|x-y|^{n-2}}f(y,u(y),\cdots)dy, \,\,\,\,\,\,\,\,\,\, \forall \,\, x\in\mathbb{R}^{n},
\end{equation}
where the Riesz potential's constants $R_{\alpha,n}:=\frac{\Gamma(\frac{n-\alpha}{2})}{\pi^{\frac{n}{2}}2^{\alpha}\Gamma(\frac{\alpha}{2})}$ for $0<\alpha<n$.

To this end, for arbitrary $R>0$, let $f_{1}(u)(x):=f(x,u(x),\cdots)$ and
\begin{equation}\label{2c2}
v_{1}^{R}(x):=\int_{B_R(0)}G^{2}_R(x,y)f_{1}(u)(y)dy,
\end{equation}
where the Green's function for $-\Delta$ on $B_R(0)$ is given by
\begin{equation}\label{Green}
  G^{2}_R(x,y)=R_{2,n}\bigg[\frac{1}{|x-y|^{n-2}}-\frac{1}{\big(|x|\cdot\big|\frac{Rx}{|x|^{2}}-\frac{y}{R}\big|\big)^{n-2}}\bigg], \,\,\,\, \text{if} \,\, x,y\in B_{R}(0),
\end{equation}
and $G^{2}_{R}(x,y)=0$ if $x$ or $y\in\mathbb{R}^{n}\setminus B_{R}(0)$. Then, we can derive that $v_{1}^{R}\in C^{2}(\mathbb{R}^{n})$ and satisfies
\begin{equation}\label{2c3}\\\begin{cases}
-\Delta v_{1}^{R}(x)=f(x,u(x),\cdots),\ \ \ \ x\in B_R(0),\\
v_{1}^{R}(x)=0,\ \ \ \ \ \ \ x\in \mathbb{R}^{n}\setminus B_R(0).
\end{cases}\end{equation}
Let $w_{1}^R(x):=(-\Delta)^{m-1+\frac{\alpha}{2}}u(x)-v_{1}^R(x)$. By Theorem \ref{Thm0}, \eqref{PDE} and \eqref{2c3}, we have $w_{1}^R\in C^{2}(\mathbb{R}^{n})$ and satisfies
\begin{equation}\label{2c4}\\\begin{cases}
-\Delta w_{1}^R(x)=0,\ \ \ \ x\in B_R(0),\\
w_{1}^{R}(x)>0, \,\,\,\,\, x\in \mathbb{R}^{n}\setminus B_R(0).
\end{cases}\end{equation}
By maximum principle, we deduce that for any $R>0$,
\begin{equation}\label{2c5}
  w_{1}^R(x)=(-\Delta)^{m-1+\frac{\alpha}{2}}u(x)-v_{1}^{R}(x)>0, \,\,\,\,\,\,\, \forall \,\, x\in\mathbb{R}^{n}.
\end{equation}
Now, for each fixed $x\in\mathbb{R}^{n}$, letting $R\rightarrow\infty$ in \eqref{2c5}, we have
\begin{equation}\label{2c6}
(-\Delta)^{m-1+\frac{\alpha}{2}}u(x)\geq\int_{\mathbb{R}^{n}}\frac{R_{2,n}}{|x-y|^{n-2}}f_{1}(u)(y)dy=:v_{1}(x)>0.
\end{equation}
Take $x=0$ in \eqref{2c6}, we get
\begin{equation}\label{2c7}
  \int_{\mathbb{R}^{n}}\frac{f(y,u(y),\cdots)}{|y|^{n-2}}dy<+\infty.
\end{equation}
One can easily observe that $v_{1}\in C^{2}(\mathbb{R}^{n})$ is a solution of
\begin{equation}\label{2c8}
-\Delta v_{1}(x)=f(x,u(x),\cdots),  \,\,\,\,\,\,\, x\in \mathbb{R}^n.
\end{equation}
Define $w_{1}(x):=(-\Delta)^{m-1+\frac{\alpha}{2}}u(x)-v_{1}(x)$. Then, by \eqref{PDE}, \eqref{2c6} and \eqref{2c8}, we have $w_{1}\in C^{2}(\mathbb{R}^{n})$ and satisfies
\begin{equation}\label{2c9}\\\begin{cases}
-\Delta w_{1}(x)=0, \,\,\,\,\,  x\in \mathbb{R}^n,\\
w_{1}(x)\geq0, \,\,\,\,\,\,  x\in \mathbb{R}^n.
\end{cases}\end{equation}
From Liouville theorem for harmonic functions, we can deduce that
\begin{equation}\label{2c10}
   w_{1}(x)=(-\Delta)^{m-1+\frac{\alpha}{2}}u(x)-v_{1}(x)\equiv C_{1}\geq0.
\end{equation}
Therefore, we have
\begin{equation}\label{2c11}
  (-\Delta)^{m-1+\frac{\alpha}{2}}u(x)=\int_{\mathbb{R}^{n}}\frac{R_{2,n}}{|x-y|^{n-2}}f(y,u(y),\cdots)dy+C_{1}=:f_{2}(u)(x)>C_{1}\geq0.
\end{equation}

Next, for arbitrary $R>0$, let
\begin{equation}\label{2c12}
v_{2}^R(x):=\int_{B_R(0)}G^{2}_R(x,y)f_{2}(u)(y)dy.
\end{equation}
Then, we can get
\begin{equation}\label{2c13}\\\begin{cases}
-\Delta v_2^R(x)=f_{2}(u)(x),\ \ x\in B_R(0),\\
v_2^R(x)=0,\ \ \ \ \ \ \ x\in \mathbb{R}^{n}\setminus B_R(0).
\end{cases}\end{equation}
Let $w_2^R(x):=(-\Delta)^{m-2+\frac{\alpha}{2}}u(x)-v_2^R(x)$. By Theorem \ref{Thm0}, \eqref{2c11} and \eqref{2c13}, we have
\begin{equation}\label{2c14}\\\begin{cases}
-\Delta w_2^R(x)=0,\ \ \ \ x\in B_R(0),\\
w_2^R(x)>0, \,\,\,\,\, x\in \mathbb{R}^{n}\setminus B_R(0).
\end{cases}\end{equation}
By maximum principle, we deduce that for any $R>0$,
\begin{equation}\label{2c15}
  w_2^R(x)=(-\Delta)^{m-2+\frac{\alpha}{2}}u(x)-v_2^{R}(x)>0, \,\,\,\,\,\,\, \forall \,\, x\in\mathbb{R}^{n}.
\end{equation}
Now, for each fixed $x\in\mathbb{R}^{n}$, letting $R\rightarrow\infty$ in \eqref{2c15}, we have
\begin{equation}\label{2c16}
(-\Delta)^{m-2+\frac{\alpha}{2}}u(x)\geq\int_{\mathbb{R}^{n}}\frac{R_{2,n}}{|x-y|^{n-2}}f_{2}(u)(y)dy=:v_{2}(x)>0.
\end{equation}
Take $x=0$ in \eqref{2c16}, we get
\begin{equation}\label{2c17}
  \int_{\mathbb{R}^{n}}\frac{C_{1}}{|y|^{n-2}}dy\leq\int_{\mathbb{R}^{n}}\frac{f_{2}(u)(y)}{|y|^{n-2}}dy<+\infty,
\end{equation}
it follows easily that $C_{1}=0$, and hence we have proved \eqref{2c1}, that is,
\begin{equation}\label{2c18}
  (-\Delta)^{m-1+\frac{\alpha}{2}}u(x)=f_{2}(u)(x)=\int_{\mathbb{R}^{n}}\frac{R_{2,n}}{|x-y|^{n-2}}f(y,u(y),\cdots)dy.
\end{equation}

One can easily observe that $v_{2}$ is a solution of
\begin{equation}\label{2c19}
-\Delta v_{2}(x)=f_{2}(u)(x),  \,\,\,\,\, x\in \mathbb{R}^n.
\end{equation}
Define $w_{2}(x):=(-\Delta)^{m-2+\frac{\alpha}{2}}u(x)-v_{2}(x)$, then it satisfies
\begin{equation}\label{2c20}\\\begin{cases}
-\Delta w_{2}(x)=0, \,\,\,\,\,  x\in \mathbb{R}^n,\\
w_{2}(x)\geq0, \,\,\,\,\,\,  x\in\mathbb{R}^n.
\end{cases}\end{equation}
From Liouville theorem for harmonic functions, we can deduce that
\begin{equation}\label{2c21}
   w_{2}(x)=(-\Delta)^{m-2+\frac{\alpha}{2}}u(x)-v_{2}(x)\equiv C_{2}\geq0.
\end{equation}
Therefore, we have proved that
\begin{equation}\label{2c22}
  (-\Delta)^{m-2+\frac{\alpha}{2}}u(x)=\int_{\mathbb{R}^{n}}\frac{R_{2,n}}{|x-y|^{n-2}}f_{2}(u)(y)dy+C_{2}=:f_{3}(u)(x)>C_{2}\geq0.
\end{equation}
By the same methods as above, we can prove that $C_{2}=0$, and hence
\begin{equation}\label{2c23}
  (-\Delta)^{m-2+\frac{\alpha}{2}}u(x)=f_{3}(u)(x)=\int_{\mathbb{R}^{n}}\frac{R_{2,n}}{|x-y|^{n-2}}f_{2}(u)(y)dy.
\end{equation}
Repeating the above argument, defining
\begin{equation}\label{2c24}
  f_{k+1}(u)(x):=\int_{\mathbb{R}^{n}}\frac{R_{2,n}}{|x-y|^{n-2}}f_{k}(u)(y)dy
\end{equation}
for $k=1,2,\cdots,m$, then by Theorem \ref{Thm0} and induction, we have
\begin{equation}\label{2c25}
  (-\Delta)^{m-k+\frac{\alpha}{2}}u(x)=f_{k+1}(u)(x)=\int_{\mathbb{R}^{n}}\frac{R_{2,n}}{|x-y|^{n-2}}f_{k}(u)(y)dy
\end{equation}
for $k=1,2,\cdots,m-1$, and
\begin{equation}\label{2c50}
  (-\Delta)^{\frac{\alpha}{2}}u(x)=\int_{\mathbb{R}^{n}}\frac{R_{2,n}}{|x-y|^{n-2}}f_{m}(u)(y)dy+C_{m}=f_{m+1}(u)(x)+C_{m}>C_{m}\geq0.
\end{equation}
For arbitrary $R>0$, let
\begin{equation}\label{2c12+}
v_{m+1}^R(x):=\int_{B_R(0)}G^{\alpha}_R(x,y)\left(f_{m+1}(u)(y)+C_{m}\right)dy,
\end{equation}
where the Green's function for $(-\Delta)^{\frac{\alpha}{2}}$ with $0<\alpha<2$ on $B_R(0)$ is given by
\begin{equation}\label{2-8c+}
G^\alpha_R(x,y):=\frac{C_{n,\alpha}}{|x-y|^{n-\alpha}}\int_{0}^{\frac{t_{R}}{s_{R}}}\frac{b^{\frac{\alpha}{2}-1}}{(1+b)^{\frac{n}{2}}}db
\,\,\,\,\,\,\,\,\, \text{if} \,\, x,y\in B_{R}(0)
\end{equation}
with $s_{R}=\frac{|x-y|^{2}}{R^{2}}$, $t_{R}=\left(1-\frac{|x|^{2}}{R^{2}}\right)\left(1-\frac{|y|^{2}}{R^{2}}\right)$, and $G^{\alpha}_{R}(x,y)=0$ if $x$ or $y\in\mathbb{R}^{n}\setminus B_{R}(0)$ (see \cite{K}). Then, we can get
\begin{equation}\label{2c13+}\\\begin{cases}
(-\Delta)^{\frac{\alpha}{2}}v_{m+1}^R(x)=f_{m+1}(u)(x)+C_{m},\ \ x\in B_R(0),\\
v_{m+1}^R(x)=0,\ \ \ \ \ \ \ x\in \mathbb{R}^{n}\setminus B_R(0).
\end{cases}\end{equation}
Let $w_{m+1}^R(x):=u(x)-v_{m+1}^R(x)$. By Theorem \ref{Thm0}, \eqref{2c50} and \eqref{2c13+}, we have
\begin{equation}\label{2c14+}\\\begin{cases}
(-\Delta)^{\frac{\alpha}{2}}w_{m+1}^R(x)=0,\ \ \ \ x\in B_R(0),\\
w_{m+1}^R(x)>0, \,\,\,\,\, x\in \mathbb{R}^{n}\setminus B_R(0).
\end{cases}\end{equation}

Now we need the following maximum principle for fractional Laplacians $(-\Delta)^{\frac{\alpha}{2}}$, which can been found in \cite{CLL,S}.
\begin{lem}\label{max} Let $\Omega$ be a bounded domain in $\mathbb{R}^{n}$ and $0<\alpha<2$. Assume that $u\in\mathcal{L}_{\alpha}\cap C^{1,1}_{loc}(\Omega)$ and is l.s.c. on $\overline{\Omega}$. If $(-\Delta)^{\frac{\alpha}{2}}u\geq 0$ in $\Omega$ and $u\geq 0$ in $\mathbb{R}^n\setminus\Omega$, then $u\geq 0$ in $\mathbb{R}^n$. Moreover, if $u=0$ at some point in $\Omega$, then $u=0$ a.e. in $\mathbb{R}^{n}$. These conclusions also hold for unbounded domain $\Omega$ if we assume further that
\[\liminf_{|x|\rightarrow\infty}u(x)\geq0.\]
\end{lem}

By Lemma \ref{max}, we can deduce immediately from \eqref{2c14+} that for any $R>0$,
\begin{equation}\label{2c15+}
  w_{m+1}^R(x)=u(x)-v_{m+1}^{R}(x)>0, \,\,\,\,\,\,\, \forall \,\, x\in\mathbb{R}^{n}.
\end{equation}
Now, for each fixed $x\in\mathbb{R}^{n}$, letting $R\rightarrow\infty$ in \eqref{2c15+}, we have
\begin{equation}\label{2c16+}
u(x)\geq\int_{\mathbb{R}^{n}}\frac{R_{\alpha,n}}{|x-y|^{n-\alpha}}\left(f_{m+1}(u)(y)+C_{m}\right)dy>0.
\end{equation}
Take $x=0$ in \eqref{2c16+}, we get
\begin{equation}\label{2c17+}
  \int_{\mathbb{R}^{n}}\frac{C_{m}}{|y|^{n-\alpha}}dy\leq\int_{\mathbb{R}^{n}}\frac{f_{m+1}(u)(y)+C_{m}}{|y|^{n-\alpha}}dy<+\infty,
\end{equation}
it follows easily that $C_{m}=0$, and hence we have
\begin{equation}\label{2c18+}
  (-\Delta)^{\frac{\alpha}{2}}u(x)=f_{m+1}(u)(x)=\int_{\mathbb{R}^{n}}\frac{R_{2,n}}{|x-y|^{n-2}}f_{m}(u)(y)dy,
\end{equation}
and
\begin{equation}\label{2c19+}
  u(x)\geq\int_{\mathbb{R}^{n}}\frac{R_{\alpha,n}}{|x-y|^{n-\alpha}}f_{m+1}(u)(y)dy.
\end{equation}
In particular, it follows from \eqref{2c25}, \eqref{2c18+} and \eqref{2c19+} that
\begin{eqnarray}\label{2c51}
  && +\infty>(-\Delta)^{m-k+\frac{\alpha}{2}}u(0)=\int_{\mathbb{R}^{n}}\frac{R_{2,n}}{|y|^{n-2}}f_{k}(u)(y)dy \\
 \nonumber &\geq& \int_{\mathbb{R}^{n}}\frac{R_{2,n}}{|y^{k}|^{n-2}}\int_{\mathbb{R}^{n}}\frac{R_{2,n}}{|y^{k}-y^{k-1}|^{n-2}}\cdots
  \int_{\mathbb{R}^{n}}\frac{R_{2,n}}{|y^{2}-y^{1}|^{n-2}}f(y^{1},u(y^{1}),\cdots)dy^{1}dy^{2}\cdots dy^{k}
\end{eqnarray}
for $k=1,2,\cdots,m$, and
\begin{eqnarray}\label{formula}
  && +\infty>u(0)\geq\int_{\mathbb{R}^{n}}\frac{R_{\alpha,n}}{|y|^{n-\alpha}}f_{m+1}(u)(y)dy \\
 \nonumber &\geq& \int_{\mathbb{R}^{n}}\frac{R_{\alpha,n}}{|y^{m+1}|^{n-\alpha}}\left(\int_{\mathbb{R}^{n}}\frac{R_{2,n}}{|y^{m+1}-y^{m}|^{n-2}}\cdots
  \int_{\mathbb{R}^{n}}\frac{R_{2,n}}{|y^{2}-y^{1}|^{n-2}}f(y^{1},\cdots)dy^{1}\cdots dy^{m}\right)dy^{m+1}.
\end{eqnarray}
From the properties of Riesz potential, for any $\alpha_{1},\alpha_{2}\in(0,n)$ such that $\alpha_{1}+\alpha_{2}\in(0,n)$, one has(see \cite{Stein})
\begin{equation}\label{2c26}
  \int_{\mathbb{R}^{n}}\frac{R_{\alpha_{1},n}}{|x-y|^{n-\alpha_{1}}}\cdot\frac{R_{\alpha_{2},n}}{|y-z|^{n-\alpha_{2}}}dy
=\frac{R_{\alpha_{1}+\alpha_{2},n}}{|x-z|^{n-(\alpha_{1}+\alpha_{2})}}.
\end{equation}
By applying \eqref{2c26} and direct calculations, we obtain that
\begin{eqnarray}\label{2c27}
  && \int_{\mathbb{R}^{n}}\frac{R_{2,n}}{|y^{m+1}-y^{m}|^{n-2}}\cdots
\int_{\mathbb{R}^{n}}\frac{R_{2,n}}{|y^{3}-y^{2}|^{n-2}}\cdot\frac{R_{2,n}}{|y^{2}-y^{1}|^{n-2}}dy^{2}\cdots dy^{m} \\
 \nonumber &=& \frac{R_{2m,n}}{|y^{m+1}-y^{1}|^{n-2m}}.
\end{eqnarray}

Now, note that $m=\frac{n-1}{2}$ and $n$ is odd, we can deduce from \eqref{formula}, \eqref{2c27} and Fubini's theorem that
\begin{eqnarray}\label{contradiction}
  +\infty&>&u(0)\geq\int_{\mathbb{R}^{n}}\frac{R_{\alpha,n}}{|y^{m+1}|^{n-\alpha}}\left(\int_{\mathbb{R}^{n}}\frac{R_{n-\alpha,n}}{|y^{m+1}-y^{1}|}
  f(y^{1},u(y^{1}),\cdots)dy^{1}\right)dy^{m+1} \\
 \nonumber &=& \frac{1}{(2\pi)^{n}}\int_{\mathbb{R}^{n}}\frac{1}{|y|^{n-\alpha}}\left(\int_{\mathbb{R}^{n}}\frac{1}{|y-z|}f(z,u(z),\cdots)dz\right)dy.
\end{eqnarray}

We will get a contradiction from \eqref{contradiction}. Indeed, if we assume that $u$ is not identically zero, then by \eqref{2-50}, $u>0$ in $\mathbb{R}^{n}$.
By the assumptions on $f$, we have, there exists a point $x_{0}\in\mathbb{R}^{n}$ such that $f>0$ at $x_{0}$. Hence by the integrability \eqref{2c7}, we have
\begin{equation}\label{2c60}
0<C_{0}:=\int_{\mathbb{R}^{n}}\frac{f(z,u(z),\cdots)}{|z|^{n-2}}dz<+\infty.
\end{equation}
For any given $|y|\geq 3$, if $|z|\geq\big(\ln|y|\big)^{-\frac{1}{n-2}}$, then one has immediately
\begin{equation}\label{2c61}
|y-z|\leq|y|+|z|\leq\left(|y|\big(\ln|y|\big)^{\frac{1}{n-2}}+1\right)|z|\leq 2|y|\big(\ln|y|\big)^{\frac{1}{n-2}}|z|.
\end{equation}
Thus it follows from \eqref{2c60} and \eqref{2c61} that, there exists a $R_{0}\geq3$ sufficiently large such that, for any $|y|\geq R_{0}$, we have
\begin{eqnarray}\label{2c63}
\int_{\mathbb{R}^{n}}\frac{f(z,u(z),\cdots)}{|y-z|}dz&\geq& \frac{1}{2|y|\ln|y|}\int_{|z|\geq\big(\ln|y|\big)^{-\frac{1}{n-2}}}\frac{f(z,u(z),\cdots)}{|z|^{n-2}}dz \\
\nonumber &\geq& \frac{1}{4|y|\ln|y|}\int_{\mathbb{R}^{n}}\frac{f(z,u(z),\cdots)}{|z|^{n-2}}dz\geq\frac{C_{0}}{4|y|\ln|y|}.
\end{eqnarray}

As a consequence, we can finally deduce from \eqref{contradiction}, \eqref{2c63} and $1\leq\alpha<2$ that
\begin{equation}\label{final}
 +\infty>u(0)\geq\frac{C_{0}}{4(2\pi)^{n}}\int_{|y|\geq R_{0}}\frac{1}{|y|^{n-\alpha+1}\ln|y|}dy=+\infty,
\end{equation}
which is a contradiction. Therefore $u\equiv0$ in $\mathbb{R}^{n}$. This proves Theorem \ref{Thm3} in Case i): $m=\frac{n-1}{2}$ and $n$ is odd.

\emph{Case ii): $m\geq\lceil\frac{n}{2}\rceil$.} Let
\begin{equation}\label{2c24-0}
  f_{k+1}(u)(x):=\int_{\mathbb{R}^{n}}\frac{R_{2,n}}{|x-y|^{n-2}}f_{k}(u)(y)dy
\end{equation}
for $k=1,2,\cdots,\lceil\frac{n}{2}\rceil$, by a quite similar way as in the proof for Case i), we can infer from Theorem \ref{Thm0} and induction that
\begin{equation}\label{2c25-0}
  (-\Delta)^{m-k+\frac{\alpha}{2}}u(x)=f_{k+1}(u)(x)=\int_{\mathbb{R}^{n}}\frac{R_{2,n}}{|x-y|^{n-2}}f_{k}(u)(y)dy
\end{equation}
for $k=1,2,\cdots,\lceil\frac{n}{2}\rceil-1$, and
\begin{equation}\label{2c50-0}
 (-\Delta)^{m-\lceil\frac{n}{2}\rceil+\frac{\alpha}{2}}u(x)\geq f_{\lceil\frac{n}{2}\rceil+1}(u)(x)=\int_{\mathbb{R}^{n}}\frac{R_{2,n}}{|x-y|^{n-2}}f_{\lceil\frac{n}{2}\rceil}(u)(y)dy.
\end{equation}
In particular, it follows from \eqref{2c25-0} and \eqref{2c50-0} that
\begin{eqnarray}\label{2c51-0}
  && +\infty>(-\Delta)^{m-k+\frac{\alpha}{2}}u(0)=\int_{\mathbb{R}^{n}}\frac{R_{2,n}}{|y|^{n-2}}f_{k}(u)(y)dy \\
 \nonumber &\geq& \int_{\mathbb{R}^{n}}\frac{R_{2,n}}{|y^{k}|^{n-2}}\int_{\mathbb{R}^{n}}\frac{R_{2,n}}{|y^{k}-y^{k-1}|^{n-2}}\cdots
  \int_{\mathbb{R}^{n}}\frac{R_{2,n}}{|y^{2}-y^{1}|^{n-2}}f(y^{1},u(y^{1}),\cdots)dy^{1}dy^{2}\cdots dy^{k}
\end{eqnarray}
for $k=1,2,\cdots,\lceil\frac{n}{2}\rceil-1$, and
\begin{eqnarray}\label{formula-0}
  && +\infty>(-\Delta)^{m-\lceil\frac{n}{2}\rceil+\frac{\alpha}{2}}u(0)\geq\int_{\mathbb{R}^{n}}\frac{R_{2,n}}{|y|^{n-2}}f_{\lceil\frac{n}{2}\rceil}(u)(y)dy
  \geq\int_{\mathbb{R}^{n}}\frac{R_{2,n}}{|y^{\lceil\frac{n}{2}\rceil}|^{n-2}} \\
 \nonumber && \qquad\times\left(\int_{\mathbb{R}^{n}}\frac{R_{2,n}}{|y^{\lceil\frac{n}{2}\rceil}-y^{\lceil\frac{n}{2}\rceil-1}|^{n-2}}\cdots
  \int_{\mathbb{R}^{n}}\frac{R_{2,n}}{|y^{2}-y^{1}|^{n-2}}f(y^{1},\cdots)dy^{1}\cdots dy^{\lceil\frac{n}{2}\rceil-1}\right)dy^{\lceil\frac{n}{2}\rceil}.
\end{eqnarray}
By applying the formula \eqref{2c26} and direct calculations, we obtain that
\begin{eqnarray}\label{2c27-0}
  && \int_{\mathbb{R}^{n}}\frac{R_{2,n}}{|y^{\lceil\frac{n}{2}\rceil}-y^{\lceil\frac{n}{2}\rceil-1}|^{n-2}}\cdots
\int_{\mathbb{R}^{n}}\frac{R_{2,n}}{|y^{3}-y^{2}|^{n-2}}\cdot\frac{R_{2,n}}{|y^{2}-y^{1}|^{n-2}}dy^{2}\cdots dy^{\lceil\frac{n}{2}\rceil-1} \\
 \nonumber &=& \frac{R_{2\lceil\frac{n}{2}\rceil-2,n}}{|y^{\lceil\frac{n}{2}\rceil}-y^{1}|^{n-2\lceil\frac{n}{2}\rceil+2}}.
\end{eqnarray}

Now, we can deduce from \eqref{formula-0}, \eqref{2c27-0} and Fubini's theorem that
\begin{eqnarray}\label{contradiction-0}
  +\infty&>&(-\Delta)^{m-\lceil\frac{n}{2}\rceil+\frac{\alpha}{2}}u(0) \\
  \nonumber &\geq&\int_{\mathbb{R}^{n}}\frac{R_{2,n}}{|y^{\lceil\frac{n}{2}\rceil}|^{n-2}}
  \left(\int_{\mathbb{R}^{n}}\frac{R_{2\lceil\frac{n}{2}\rceil-2,n}}{|y^{\lceil\frac{n}{2}\rceil}-y^{1}|^{n-2\lceil\frac{n}{2}\rceil+2}}f(y^{1},u(y^{1}),\cdots)dy^{1}\right)dy^{\lceil\frac{n}{2}\rceil} \\
 \nonumber &=& C_{n}\int_{\mathbb{R}^{n}}\frac{1}{|y|^{n-2}}\left(\int_{\mathbb{R}^{n}}\frac{1}{|y-z|^{n-2\lceil\frac{n}{2}\rceil+2}}f(z,u(z),\cdots)dz\right)dy.
\end{eqnarray}

We will get a contradiction from \eqref{contradiction-0}. To do this, let $\tau(n):=n-2\lceil\frac{n}{2}\rceil+2\in\{1,2\}$, then it follows from \eqref{2c60} and \eqref{2c61} that, there exists a $R_{0}\geq3$ sufficiently large such that, for any $|y|\geq R_{0}$,
\begin{eqnarray}\label{2c63-0}
&&\int_{\mathbb{R}^{n}}\frac{1}{|y-z|^{\tau(n)}}f(z,u(z),\cdots)dz \\
\nonumber &\geq& \frac{1}{2^{\tau(n)}|y|^{\tau(n)}\ln|y|}\int_{|z|\geq\big(\ln|y|\big)^{-\frac{1}{n-2}}}\frac{1}{|z|^{n-2}}f(z,u(z),\cdots)dz \\
\nonumber &\geq& \frac{1}{2^{\tau(n)+1}|y|^{\tau(n)}\ln|y|}\int_{\mathbb{R}^{n}}\frac{f(z,u(z),\cdots)}{|z|^{n-2}}dz\geq\frac{C_{0}}{2^{\tau(n)+1}|y|^{\tau(n)}\ln|y|}.
\end{eqnarray}
Therefore, we can finally deduce from \eqref{contradiction-0} and \eqref{2c63-0} that
\begin{equation}\label{final-0}
 +\infty>(-\Delta)^{m-\lceil\frac{n}{2}\rceil+\frac{\alpha}{2}}u(0)\geq\frac{C_{0}C_{n}}{2^{\tau(n)+1}}\int_{|y|\geq R_{0}}\frac{1}{|y|^{n-2+\tau(n)}\ln|y|}dy=+\infty,
\end{equation}
which is a contradiction again. Therefore, $u\equiv0$ in $\mathbb{R}^{n}$ in Case ii): $m\geq \lceil\frac{n}{2}\rceil$.

This concludes our proof of Theorem \ref{Thm3}.

\section{Appendix: A characterization for $\alpha$-harmonic functions via averages}

One can observe from the proof of Theorem \ref{Thm0} and \ref{Thm1} that, the average $\int_{R}^{+\infty}\frac{R^{\alpha}}{r(r^{2}-R^{2})^{\frac{\alpha}{2}}}\overline{u}(r)dr$ plays an basic and important role for the nonlocal fractional Laplacians $(-\Delta)^{\frac{\alpha}{2}}$ ($0<\alpha<2$), which is similar as $\overline{u}(r)$ for the Laplacian $-\Delta$. In this appendix, we will characterize the $\alpha$-harmonic functions by using the averages $\int_{R}^{+\infty}\frac{R^{\alpha}}{r(r^{2}-R^{2})^{\frac{\alpha}{2}}}\overline{u}(r)dr$ and deduce some important properties for $\alpha$-harmonic functions.

Let $\Omega\subseteq\mathbb{R}^{n}$ be a (bounded or unbounded) domain. We have the following characterization for $\alpha$-harmonic functions in $\Omega$.
\begin{thm}\label{ThmA1}
Assume $0<\alpha<2$. Let $u\in C^{[\alpha],\{\alpha\}+\epsilon}_{loc}(\Omega)\cap \mathcal{L}_{\alpha}(\mathbb{R}^{n})$ (with $\epsilon>0$ arbitrarily small) satisfy $(-\Delta)^{\frac{\alpha}{2}}u\geq0$ ($\leq0$) in $\Omega$, then for any ball $B=B_{R}(y)\subset\subset\Omega$, we have
\begin{equation}\label{A0}
  u(y)\geq (\leq) C_{n,\alpha}\int_{R}^{+\infty}\frac{R^{\alpha}}{r(r^{2}-R^{2})^{\frac{\alpha}{2}}}\overline{u}(r)dr,
\end{equation}
where $C_{n,\alpha}:=\frac{\Gamma(\frac{n}{2})}{\pi^{\frac{n}{2}+1}}\sin\frac{\pi\alpha}{2}$ and $\overline{u}(r)$ denotes the spherical average of $u$ w.r.t. $y$. Furthermore, $(-\Delta)^{\frac{\alpha}{2}}u=0$ in $\Omega$ if and only if
\begin{equation}\label{A1}
  u(y)=C_{n,\alpha}\int_{R}^{+\infty}\frac{R^{\alpha}}{r(r^{2}-R^{2})^{\frac{\alpha}{2}}}\overline{u}(r)dr
\end{equation}
for any ball $B=B_{R}(y)\subset\subset\Omega$.
\end{thm}
\begin{proof}
Suppose $(-\Delta)^{\frac{\alpha}{2}}u\geq0$ in $\Omega$, then it follows from Green-Poisson integral representation formula that, for any ball $B=B_{R}(y)\subset\subset\Omega$,
\begin{eqnarray}\label{A2}
  &&u(y)=\int_{B_{R}(y)}G^\alpha_R(y,z)(-\Delta)^{\frac{\alpha}{2}}u(z)dz+\int_{|z-y|>R}P^{\alpha}_{R}(y,z)u(z)dz \\
  \nonumber &&\qquad\,\,\,=\int_{B_{R}(y)}\frac{\widetilde{C}_{n,\alpha}}{|z-y|^{n-\alpha}}\left(\int_{0}^{\frac{R^{2}}{|z-y|^{2}}-1}\frac{b^{\frac{\alpha}{2}-1}}{(1+b)^{\frac{n}{2}}}db\right)
  (-\Delta)^{\frac{\alpha}{2}}u(z)dz \\
  \nonumber &&\qquad\quad\,\,\,+C_{n,\alpha}\int_{|z-y|>R}\frac{R^{\alpha}}{(|z-y|^{2}-R^{2})^{\frac{\alpha}{2}}}\cdot\frac{u(z)}{|z-y|^{n}}dz \\
 \nonumber &&\qquad\,\,\,\geq C_{n,\alpha}\int_{R}^{+\infty}\frac{R^{\alpha}}{r(r^{2}-R^{2})^{\frac{\alpha}{2}}}\overline{u}(r)dr.
\end{eqnarray}
If $(-\Delta)^{\frac{\alpha}{2}}u\leq0$ in $\Omega$, then \eqref{A0} can be derived in entirely similar way.

Now assume that $u$ satisfies the average property \eqref{A1} for any ball $B=B_{R}(y)\subset\subset\Omega$, we will show that $u$ is $\alpha$-harmonic in $\Omega$. To this end, for any ball $B=B_{R}(y)\subset\subset\Omega$, let us define
\begin{equation}\label{A3}
  h(x):=C_{n,\alpha}\int_{|z-y|>R}\left(\frac{R^{2}-|x-y|^{2}}{|z-y|^{2}-R^{2}}\right)^{\frac{\alpha}{2}}\cdot\frac{u(z)}{|x-z|^{n}}dz, \qquad \forall \,\, x\in B_{R}(y),
\end{equation}
and $h(x):=u(x)$ if $|x-y|\geq R$. Then it follows that $(-\Delta)^{\frac{\alpha}{2}}h=0$ in $B_{R}(y)$ and hence satisfies the average property \eqref{A1} for any ball $B\subset\subset B_{R}(y)$. Define $w:=u-h$, then $w(x)=0$ if $|x-y|\geq R$ and $w$ satisfies the average property \eqref{A1} for any ball $B\subset\subset B_{R}(y)$. Our aim is to show that $w=0$ in $B_{R}(y)$.

Indeed, suppose there exists a point $\bar{x}\in B_{R}(y)$ such that $w(\bar{x})=M:=\max_{x\in B_{R}(y)}w(x)>0$, then the average property \eqref{A1} implies that, for any ball $B_{\bar{R}}(\bar{x})\subset\subset B_{R}(y)$,
\begin{equation}\label{A4}
  0=w(\bar{x})-M=C_{n,\alpha}\int_{\bar{R}}^{+\infty}\frac{\bar{R}^{\alpha}}{r(r^{2}-\bar{R}^{2})^{\frac{\alpha}{2}}}\overline{w-M}(r)dr<0,
\end{equation}
where $\overline{w-M}(r)$ denotes the spherical average of $w-M$ w.r.t. $\bar{x}$. This is absurd, thus $w\leq0$ in $B_{R}(y)$. Similarly, we can show that $m:=\min_{x\in B_{R}(y)}w(x)=0$. Therefore, $w=u-h=0$ in $B_{R}(y)$ and hence $(-\Delta)^{\frac{\alpha}{2}}u=0$ in $B_{R}(y)$. Since $B=B_{R}(y)\subset\subset\Omega$ is arbitrary, we deduce that $(-\Delta)^{\frac{\alpha}{2}}u=0$ in $\Omega$. This completes our proof of Theorem \ref{ThmA1}.
\end{proof}

As an immediate consequence of Theorem \ref{ThmA1}, we have the following Theorem.
\begin{thm}\label{ThmA2}
Suppose $\{u_{n}\}_{n\geq1}$ is a sequence of $\alpha$-harmonic functions in $\Omega$ and $u_{n}\rightrightarrows u$ in $\mathbb{R}^{n}$, then $u$ is $\alpha$-harmonic in $\Omega$.
\end{thm}

\begin{rem}\label{remark8}
The Harnack inequalities (see \cite{Landkof}) and Liouville theorems for $\alpha$-harmonic functions (see \cite{BKN,ZCCY}), and Maximal principles for fractional Laplacians $(-\Delta)^{\frac{\alpha}{2}}$ (see \cite{CLL,S}) can also be deduced directly from Theorem \ref{ThmA1}, we omit the proofs here and leave the details to readers.
\end{rem}

\end{document}